\documentclass[12pt]{amsart}
\usepackage{amsmath,amsthm,latexsym,amsfonts,amssymb,mathrsfs,color}
\usepackage{graphicx,verbatim}

\newcommand\delb[1]{}

\def\todown{\searrow}

\numberwithin{equation}{section}

\newtheorem{theorem}{Theorem}[section]
\newtheorem{lemma}[theorem]{Lemma}
\newtheorem{corollary}[theorem]{Corollary}
\newtheorem{proposition}[theorem]{Proposition}
\newtheorem{definition}[theorem]{Definition}
\newtheorem{assumption}[theorem]{Assumption}

\theoremstyle{remark}
\newtheorem{remark}[theorem]{Remark}

\newcommand{\be}{\begin{equation}}
\newcommand{\ee}{\end{equation}}
\newcommand{\ba}{\begin{array}}
\newcommand{\ea}{\end{array}}
\newcommand{\beas}{\begin{eqnarray*}}
\newcommand{\eeas}{\end{eqnarray*}}
\newcommand{\bea}{\begin{eqnarray}}
\newcommand{\eea}{\end{eqnarray}}

\newcommand{\Leray}{{\mathrm{P}} \, }
\newcommand{\Curl}{{\mathrm{Curl}} \, }
\newcommand{\curln}{{\mathrm{curl}}\, } 
\newcommand{\curlb}{{\boldsymbol{\mathrm{curl}}}\, } 

\newcommand{\Div}{\mathrm{div} \, }
\newcommand{\Grad}{\nabla}
\newcommand{\DDelta}{\boldsymbol{\mathrm{\Delta}} \, }
\newcommand{\A}{\boldsymbol{\mathrm{A}}  }
\newcommand{\B}{\boldsymbol{\mathrm{B}}  }

\newcommand{\CC}{\boldsymbol{\mathrm{C}}  }
\newcommand{\bL}{\boldsymbol{\mathrm{L}}  }
\newcommand{\loc}{\mathrm{loc}}

\newcommand{\nnabla}{\boldsymbol{\mathrm{\nabla}} \, }

\newcommand{\ve}{\mathbf{e}}

\newcommand{\vf}{{\bf f}}
\newcommand{\vg}{{\bf g}}

\newcommand{\vu}{\mathbf{ u}}

\newcommand{\vv}{{\bf v}}
\newcommand{\vw}{{\bf w}}
\newcommand{\vx}{{\bf x}}
\newcommand{\vy}{{\bf y}}
\newcommand{\vY}{{\bf Y}}

\newcommand{\ovz}{{\bf z}}
\newcommand{\x}{\mathbf{x}}
\newcommand{\xb}{\mathbf{x}}
\newcommand{\yb}{\mathbf{y}}

\newcommand{\vphi}{{\bf \phi}}

\newcommand{\xh}{\widehat{\x}}
\newcommand{\vomega}{\boldsymbol{\mathrm{\omega}} \, }
\newcommand{\bbP}{\mathbb{P}}
\newcommand{\bS}{\mathbb{S}^2}
\newcommand{\calA}{\mathcal A}
\newcommand{\calB}{\mathcal B}

\newcommand{\calD}{\mathcal D}
\newcommand{\calF}{\mathcal F}
\newcommand{\calG}{\mathcal G}
\newcommand{\calH}{\mathcal H}
\newcommand{\calV}{\mathcal V}
\newcommand{\calL}{\mathcal L}

\newcommand{\calX}{\mathcal X}
\newcommand{\frakD}{\mathfrak D}
\newcommand{\frakR}{\mathfrak R}
\newcommand{\frakT}{\mathfrak T}

\newcommand{\tvu}{{\tilde{\vu}}}

\newcommand{\tvv}{{\tilde{\vv}}}

\newcommand{\hH}{\mathbb{H}}
\newcommand{\nN}{\mathbb{N}}
\newcommand{\E}{\mathbb{E}}
\newcommand{\R}{\mathbb{R}}

\newcommand{\zZ}{\mathbb{Z}}

\newcommand{\nk}{ {n^{(k)}} }

\newcommand\del[1]{}
\newcommand\delc[1]{} 
\newcommand\comc[1]{ {\color{red} \small{#1 }}} 

\title[Stochastic Navier--Stokes on spheres]
{Random attractors for the stochastic Navier--Stokes equations
on the 2D unit sphere}
\author{Z.~Brze\'{z}niak}
\address{%
Department of Mathematics \\
The University of York \\
Heslington, York, Y010 5DD \\
United Kingdom}

\email{zdzislaw.brzezniak@york.ac.uk}
\author{B.~Goldys}
\address{%
School of Mathematics and Statistics \\
The University of Sydney, NSW 2006 \\
Australia}

\email{beniamin.goldys@sydney.edu.au}

\author{Q. T. Le Gia}
\address{%
School of Mathematics and Statistics \\
University of New South Wales \\
Sydney, NSW 2052, Australia}

\email{qlegia@unsw.edu.au}
\thanks{This work was completed with the support of
Australian Research Council. The authors would like to thank
an anonymous referee for helpful comments to improve the
presentation of the paper.}

\subjclass{Primary 35B41; Secondary 35Q35}

\keywords{random attractors, energy method, asymptotically
compact random dynamical systems, stochastic Navier-Stokes, unit sphere}

\date{\today}

\begin{document}
\begin{abstract}
In this paper we prove the existence of random attractors for
the Navier--Stokes equations on 2 dimensional sphere under random forcing irregular
in space and time. We also deduce the existence of an invariant measure.
\end{abstract}
\maketitle
\section{Introduction}

Complex three dimensional flows in the atmosphere and oceans are modelled assuming that
the Earth's surface is an approximate sphere. Then it is natural to model the global
atmospheric circulation on Earth (and large
planets) using the Navier-Stokes equations (NSE)
on 2-dimensional sphere coupled to classical thermodynamics~\cite{atmosp_model_books}.
This approach is relevant for geophysical flow modeling.

Many authors have studied the deterministic NSEs on the unit sphere. Notably,
Il'in and Filatov \cite{IliFil89,Ili91} considered the existence and uniqueness of solutions
to these equations and estimated the Hausdorff
dimension of their global attractors \cite{Ili94}. Temam and Wang \cite{TemWan93}
considered the inertial forms of NSEs on sphere while
Temam and Ziane \cite{TemZia97}, see also \cite{Avez+Bamberger_1978},  proved that the
NSEs on a 2-dimensional sphere is a limit
of NSEs defined on a spherical shell \cite{TemZia97}. In other directions, Cao, Rammaha and
Titi \cite{CaoRamTit99} proved the Gevrey regularity of the solution and found an upper
bound on the asymptotic degrees of
freedom for the long-time dynamics.

Concerning the numerical simulation of the deterministic NSEs on
sphere, Fengler and Freeden \cite{FenFre05} obtained some impressive
numerical results using the spectral method, while the numerical analysis of a pseudo-
spectral method for these equations has been carried out in Ganesh, Le Gia and Sloan
in \cite{GanLeGSlo11}.

In our earlier paper \cite{BrzGolLeG15} we analysed
the Navier-Stokes equations on the 2-dimensional sphere with Gaussian random forcing. We proved the existence
and uniqueness of solutions and continuous dependence on data in various topologies.  We also studied
qualitative properties of the stochastic NSEs on the unit sphere in the context of random
dynamical systems.

Building on those preliminary studies, in the current  paper, we prove the existence of
random attractors for the stochastic NSEs on the 2-dimensional unit sphere.
Let us recall here that, given a probability space, a random
attractor is a compact random set, invariant for the associated
random dynamical system and attracting every bounded random set in its basis of
attraction (see Definition \ref{randomattractor}).

In the area of SPDEs the notions of random
and pullback attractors were introduced by
Brze{\'z}niak et al. in \cite{Brz+Cap+Fland_1993}, and  by Crauel and
Flandoli in \cite{Crauel+F_1994}. These concepts have been later used to obtain
crucial information on the asymptotic behaviour of random
(Brze{\'z}niak et al. \cite{Brz+Cap+Fland_1993}), stochastic
(Arnold \cite{Arnold_1998}, Crauel and Flandoli
\cite{Crauel+F_1994}, Crauel \cite{Crauel_1995},Flandoli and Schmalfuss \cite{Flandoli-Schmalfuss_1996}) and
non-autonomous PDEs (Schmalfuss \cite{Schmalfuss}, Kloeden and
Schmalfuss \cite{Kloeden}, Caraballo et al.
\cite{Caraballo_L_R_2006}).

We do not know if our system is dissipative in $H^1$. Therefore,  despite
the fact that the embedding $H^1 \hookrightarrow L^2$ is compact, the asymptotic
compactness approach seems to be the only method available in the $L^2$-setting
to yield the existence of an attractor, hence of an invariant measure.

The paper is organised as follows. In Section~\ref{sec:NSEs on S}, we recall the
relevant properties of the deterministic NSEs on the unit sphere,
outline key function spaces, and recall the weak formulation of these equations.
In Section~\ref{sec:sNSEs on S}, we define the stochastic NSEs on the unit sphere. The stochastic NSEs is decomposed into an Ornstein-Uhlenbeck process
and a deterministic NSEs with random forcing. First we construct a stationary
solution onf the Ornstein-Uhlenbeck process (associated with the linear part
of the stochastic NSEs) and then identify a shift-invariant subset of full
measure satisfying a strong law of large numbers.
We then review the key existence and uniqueness results obtained in \cite{BrzGolLeG15}.
In Section~\ref{sec:attractor} we prove the existence of a random attractor of
the stochastic NSEs on the 2-d sphere, which is the main result of the paper.
In doing so, we present Lemma~\ref{newlemma6_5}, which is a corrected
version of \cite[Lemma 6.5]{BrzGolLeG15}. Based on the asymptotic results in the lemma,
 a new class of functions $\mathfrak{R}$ is defined in Definition~\ref{def classR}. The class
$\mathfrak{DR}$ of all closed and bounded random sets is then defined using functions
in the class $\mathfrak{R}$.
The main results are given in Theorem~\ref{thm-main}, which asserts that the
random dynamical system generated $\varphi$ by the NSEs
on the unit sphere is $\mathfrak{DR}$-asymptotically compact. Hence, in view of a result
on existence of a random attractor (Theorem~\ref{Teorema1}), the existence of a random
attractor of $\varphi$ is deduced.

The paper
is concluded with a simple proof of the existence of an invariant measure and some comments
on the question of its uniqueness.
\par
In our paper a special attention is given to the noise with low space regularity. While many
works on random attractors consider only finite dimensional noise, we follow here the approach
from Brze{\'z}niak et al \cite{Brzetal10} and consider an infinite dimensional driving
Wiener process with minimal assumptions on its Cameron-Martin space
(known also as the Reproducing Kernel Hilbert Space),
see Remark \ref{rem-radon} and the Introduction to \cite{Brzetal10} for motivation.

\section{The Navier--Stokes equations on a rotating unit sphere}
\label{sec:NSEs on S}
The sphere is a very special example of a compact Riemannian manifold without boundary hence one
could recall all the classical tools from differential geometry developed for such manifolds.
However we have decided to follow a different path of using the polar coordinates and defining all such objects directly.

Our presentation here is a self-contained version of an analogous introductory
section from our companion paper \cite{BrzGolLeG15}. A reader who is familiar with the last reference
can skip reading this section.
\subsection{Preliminaries}
By $\bS$ we will denote the standard $2$-dimensional unit sphere,
i,e, a subset of the Euclidean space $\R^3$ described by
\begin{equation}\label{eqn-2.01}
 \bS=\{ \xb=(x_1,x_2,x_3) \in \R^3: \vert \xb \vert^2 = x_1^2+x_2^2 +x_3^2 = 1\}.
\end{equation}

Let us now define the surface gradient  $\nnabla$ and $\curlb$ operators acting on tangent vector fields and the surface gradient acting on scalar functions, all  with respect to scalar product in the tangent spaces $T_{\xb} \bS$ inherited from $\R^3$.

Suppose that $\vu$ and $\vv$ are two tangent vector fields on $\bS$ and $f: \bS \rightarrow \R$, all of $C^1$ class.   By some classical results,
see for instance \cite{DubNovFom86} or \cite[Definition 3.31]{Dri04},
there exist a neighbourhood $U$ of $\bS$ in $\R^3$ and vector fields $\tilde{\vu}:U\rightarrow \R^3$,
$\tilde{\vv}:U \rightarrow \R^3$  and a function $\tilde{f}:U \rightarrow \R^3$ such that $\tilde{\vu}|_{\bS} = \vu$, $\tilde{\vv}|_{\bS} = \vv$ and
 $\tilde{f}|_{\bS} = f$.
Then we put, for $\xb \in \bS$,
\begin{eqnarray}\label{eqn-nablauu}
  (\nnabla_\vv \vu)(\xb) &=& \pi_{\xb}
  \Big( \sum_{i=1}^3 \tilde{\vv}_i(\xb) \partial_i \tvu (\xb) \Big)=\pi_{\xb}
  \big(  ( \tilde{\vv}(\xb) \cdot\tilde{\nabla})  \tvu (\xb) \big),
\\
\label{eqn-curlb}
(\curlb \vu)(\xb) &=& (\mathrm{id} -\pi_{\xb})\big( (\tilde{\nabla}\times  \tvu) (\xb) \big)=\big( \xb \cdot  (\tilde{\nabla}\times  \tvu) (\xb) \big)\,\xb,
\\
\label{eqn-gradient f}
(\nabla f)(\xb) &=& \pi_{\xb}
  \big(  \tilde{\nabla} \tilde{f}  (\xb)\big),
\end{eqnarray}
where $\tilde{\nabla}$ is the gradient in $\R^3$ and, for $\xb \in \bS$,
the map $\pi_{\xb} : \R^3 \rightarrow T_{\xb} \bS$ is the orthogonal projection, i.e.
\begin{equation}\label{eqn-pi}
\pi_{\xb}:\R^3\ni \yb\mapsto \yb-(\xb \cdot \yb)\xb=-\xb \times( \xb \times \yb)\in  T_{{\xb}}\bS.
\end{equation}
Let us point out that the definitions of $(\nnabla_\vv \vu)$ and $\curlb \vu $ above are independent of the choice of the extensions $\tilde{\vu}$ and $\tilde{\vv}$.
In the former case, this can be shown either using a general approach from the references above or, as in our companion paper \cite{BrzGolLeG15} by exploiting
a  well known formula for the $\R^3$-vector product\footnote{
\[
\mathbf{a}\times(\mathbf{b}\times \mathbf{c})=( \mathbf{a} \cdot \mathbf{c})  \mathbf{b}- ( \mathbf{a} \cdot \mathbf{b}) \mathbf{c}, \;\; \mathbf{a},\mathbf{b}, \mathbf{c}\in \R^3.
\]}
to get
\begin{equation}\label{short_nonlin0}
  (\tvu \cdot \tilde{\nabla}) \tvu  =  \tilde{\nabla} \frac{|\tvu|^2}{2} - \tvu \times (\tilde{\nabla} \times \tvu).
\end{equation}
\delc{
It follows that, using the notation above, that the restrictions to $\bS$ of  $\tvu \times (\tilde{\nabla} \times \tvu)$ and
$\tilde{\tvu} \times (\tilde{\nabla} \times \tilde{\tvu})$ coincide.}

If follows from the definition \eqref{eqn-curlb} that
$\curlb \vu$ is a normal vector field on $\bS$, i.e. $\curlb \vu (\xb)\perp T_x\bS$ for every $\xb \in \bS$. Since the co-dimension of $T_{\xb}\bS$ in $\mathbb{R}^3$ is equal to $1$, this normal vector field can be  identified  with a scalar function on $\bS$ denoted by $\curln \vu$ by
\[
(\curln \vu (\xb)) \xb= [\curlb \vu](\xb), \;\;\; \xb \in\bS.
\]
Note that it follows that

\begin{equation} \label{equ:defcurln}
\curln \vu (\xb):= \xb \cdot  (\tilde{\nabla}\times  \tvu) (\xb) , \;\;\; \xb \in\bS.
\end{equation}

%
\begin{lemma}\label{lem-vector product}
If  $\tvu$ and $\tvv$ are $\mathbb{R}^3$-valued vector fields on  $\bS$, and $\vu$ and $\vv$ are tangent vector fields on  $\bS$, defined by    $\vu(\xb) = \pi_{\xb} (\tvu((\xb)))$ and $\tvv(\xb) = \pi_{\xb}(\tvv(\xb))$, $\xb \in {\bS}$, then
the following identity holds
\begin{equation}\label{equ:piuv}
\pi_{\xb} (\tvu(\xb) \times \tvv(\xb)) =
\vu(\xb) \times ( (\xb \cdot \vv(\xb)) \xb) +
(\xb \cdot \vu(\xb)) \xb \times \vv(\xb),\;\;\;    \;\; \xb\in \bS.
\end{equation}
\end{lemma}
\begin{proof} Let us fix $\xb \in {\bS}$. Then
we can decompose vectors $\tvu=\tvu(\xb)$ and $\tvu=\tvv(\xb)$ into the tangential $\vu=\vu(\xb) \in T_{\xb} \bS$ and $\vv=\vv(\xb) \in T_{\xb} \bS$, and the normal
component as follows
\begin{align*}
  \tvu &= \vu + \vu^\perp \quad
  \text{ with }\delc{\vu \in T_{\xb} \bS,}
  \;\; \vu^\perp = (\vu \cdot \xb ) \xb \\
  \tvv &= \vv + \vv^\perp \quad
  \text{ with }\delc{\vv \in T_{\xb} \bS,}
  \;\; \vv^\perp = (\vv \cdot \xb ) \xb
\end{align*}
Thus, as $\vu \times \vv$ is normal to $T_{\xb} \bS$ so that  $\pi_{\xb}(\vu \times \vv) = 0$, and $\vu \times \vv^\perp, \vu^\perp \times \vv \in T_{\xb} \bS$,
 we infer that
\[
 \pi_{\xb}\big(\tvu \times \tvv\big) = \pi_{\xb}\big(\vu \times \vv + \vu \times \vv^\perp + \vu^\perp \times \vv\big) =\vu \times \vv^\perp + \vu^\perp \times \vv.
\]
 Hence the lemma is proved.
\end{proof}

Suppose now again that $\vu$ is a  tangent vector fields on $\bS$ and $\tilde{\vu}$ is a $\R^3$-valued vector field defined in some
 neighbourhood $U$ of $\bS$ in $\R^3$  such that $\tilde{\vu}|_{\bS} = \vu$. Applying formula \eqref{equ:piuv} to  the vector fields\footnote{(or rather their respective restrictions to the sphere $\bS$)} $\tilde{\vu}$ and
$\tvv = \tilde{\nabla} \times \tvu$, since also $(\tvu \cdot \xb)=0$ we get
\begin{eqnarray}\label{equ:conc1}
\pi_{\xb}(\tvu \times (\tilde{\nabla} \times \tvu))
&=& \vu \times ( (\xb \cdot (\tilde{\nabla} \times \tvu)) \xb)
 + (\tvu \cdot \xb) \xb \times (\tilde{\nabla} \times \tvu)
\\
\delc{With a tangential vector field $\vu$, since the normal component
is zero, \eqref{equ:conc1} is reduced to}
\label{equ:conc3}
\delc{\pi_{\xb}(\tvu \times (\tilde{\nabla} \times \tvu))}
&=& \vu \times \big( (\xb \cdot (\tilde{\nabla} \times \tvu)) \xb\big)
=(\xb \cdot (\tilde{\nabla} \times \tvu))\big(\vu \times  \xb \big), \;\;\; \xb \in\bS.
\end{eqnarray}

\delc{
\comc{I don't follow why we use a word "So" here. I would rather put here "Let". What do you think?}
So, for a tangential vector field $\vu$ on $\bS$, by $\curln \vu$ we will understand an $\mathbb{R}$-valued function defined by
\begin{equation*}
[\curln \vu](\xb) := \xb \cdot [(\tilde{\nabla} \times \tvu)(\xb)], \;\;\; \xb \in\bS.
\end{equation*}
}

Hence by formulae  \eqref{equ:conc3}, \eqref{equ:defcurln} and the above definitions
we obtain
\begin{equation}
  \pi_{\xb} [\tvu \times (\tilde{\nabla} \times \tvu)] (\xb)=  [\vu(\xb) \times \xb] \; {\curln}\vu(\xb) = \vu(\xb) \times \curlb \vu(\xb) \,   \;\; \xb \in \mathbb{S}^2, \label{eqn:star}
\end{equation}

Here, we use the following notation. Given another  tangential vector field $\vv$ on $\bS$, we will denote by $ \vv \times \curlb \vu$ a tangential vector field defined
as the $\mathbb{R}^3$ vector product of vectors,  tangent $ \vv$ and  normal $\curlb \vu$, i.e.
\begin{equation}
\label{eqn-vxcurlu}
 [\vv \times \curlb \vu](\xb):= \vv (\xb) \times \curlb \vu (\xb)), \;\;\; \xb \in\bS.
 \end{equation}
\delc{ Note that then we have
\[
 [\vv \times \curlb \vu](\xb) =  \vv (\xb) \times \xb \;(\curln \vu (\xb)), \;\;\; \xb \in\bS.
\]
}
 Thus from the above and \eqref{short_nonlin1} we infer that
\begin{equation}\label{short_nonlin1}
        \nnabla_\vu \vu = \Grad \frac{|\vu|^2}{2} - \vu \times {\curln}\vu.
\end{equation}

We will use the classical spherical coordinates to describe (in a non-unique way) the points on the sphere $\bS$

\begin{equation}\label{eqn-2.02}
  \xb =\xh(\theta,\phi) =(\sin\theta\cos\phi,\sin\theta\sin\phi,\cos\theta),
  \quad 0\le \theta \le \pi,\; 0\le \phi \leq 2\pi.
\end{equation}

With a bit of ambiguity, if $\xb =\xh(\theta,\phi)$ as in \eqref{eqn-2.02}, the
angles $\theta$ and $\phi$ will be denoted by $\theta(\xb)$ and $\phi(\xb)$,
or just $\theta$ and $\phi$.

For $(\theta,\phi) \in [0,\pi] \times [0,2\pi)$, by
$\ve_\theta=\ve_\theta(\theta,\phi)$
and $\ve_\phi=\ve_\phi(\theta,\phi)$ we will denote an orthonormal basis in the tangent
plane $T_{\xb} \bS$, where $\xb = \xh(\theta,\phi)$, defined by

\begin{equation}\label{eqn-2.03}
\ve_\theta = (\cos\theta\cos\phi,\cos\theta\sin\phi,-\sin\theta), \quad
\ve_\phi = (-\sin\phi,\cos\phi,0).
\end{equation}

If  $f: \bS \rightarrow \R$ is $C^1$ a class function,  then we can show that the surface gradient of $f$ has the following form
\begin{equation}\label{eqn-2.04}
 \Grad f =  \frac{\partial \hat{f}}{\partial \theta} \ve_\theta +
           \frac{1}{\sin \theta} \frac{\partial \hat{f}}{\partial \phi} \ve_\phi,
\end{equation}
where $\xb = \xh(\theta,\phi)$ and $\hat{f}: [0,\pi] \times [0,2\pi) \rightarrow \R$
is such that $f(\xh(\theta,\phi)) = \hat{f}(\theta,\phi)$ for all $(\theta,\phi)$.
In what follows, we will often not distinguish between functions $f$ and $\hat{f}$ and
use the notation $f$ for both.

Similarly, if $\vu$ is a (tangential) vector field on $\bS$ which can be written in
a form $\vu = (u_\theta,u_\phi)$ with respect to the (moving) basis $\ve_{\theta}, \ve_{\phi}$,
that is
\[
  \vu(\theta,\phi) = u_\theta(\theta,\phi) \ve_\theta(\theta,\phi)+u_\phi(\theta,\phi) \ve_\phi(\theta,\phi)
\]
we define the surface divergence with respect to the surface area on $\bS$ by the formula
\begin{equation}\label{eqn-2.05}
 \Div \vu = \frac{1}{\sin\theta} \left( \frac{\partial}{\partial \theta} (u_\theta \sin \theta) +
                                        \frac{\partial}{ \partial \phi} u_\phi \right).
\end{equation}

\delc{
Let us  note that $(\xb \cdot (\tilde{\nabla} \times \tvu)) \xb$ is just
the orthogonal projection of the vector $\tilde{\nabla} \times \tvu$ onto the normal component of $T_{\xb} \bS$. Furthermore, }
With slight abuse of notation, for $\xb\in \bS$,
\begin{equation}\label{equ:conc2}
\begin{aligned}
\curln \vu(\xb)=\xb \cdot (\tilde{\nabla} \times \tvu) &=
  x_1(\partial_2 u_3 - \partial_3 u_2)
+ x_2(\partial_3 u_1 - \partial_1 u_3)
+ x_3(\partial_1 u_2 - \partial_2 u_1)\\
&=\partial_1(x_3 u_2 - x_2 u_3) + \partial_2(x_1 u_3 - x_3 u_1)
+ \partial_3(x_2 u_1 - x_1 u_2) \\
& = \mathrm{div}(\tvu \times \xb) = - \mathrm{div}(\xb \times \tvu).
\end{aligned}
\end{equation}

Finally, see \cite{Ili91},
\delc{Assume that  $\vu$ is  a tangent vector field on $\bS$, and} if  $f$ a scalar function on  $\bS$, then
we define a tangent vector field $\Curl f$ by
\begin{equation}\label{Curl2}
[\Curl f](\xb) = -\xb \times \Grad f (\xb),  \quad
\delc{[\curln \vu](\xb) =  -[\Div(\xb \times \vu)](\xb),}  \;\; \xb\in \bS
\end{equation}

\delc{Moreover, given two tangential vector fields $\vu$ and $\vv$, the tangential vector field
\begin{equation}\label{eqn-v curl u}
[\vu(\xb) \times \xb]{\curln}\vv(\xb) \,, \;\; \xb\in \mathbb{S}^2
\end{equation}
 will be denoted by $\vu \times {\curln}\del{_{\xb}} \vv$.}

The surface diffusion operator acting on  tangential
vector fields  on $\bS$ is  denoted by $\DDelta$
(known as the vector Laplace-Beltrami or Laplace-de Rham operator)
and  is defined as
\begin{equation}\label{deRham}
 \DDelta \vu = \Grad \Div \vu - \Curl \curln \vu.
\end{equation}
Using \eqref{equ:conc2} and \eqref{Curl2}, one can derive the following relations connecting the above operators:
\begin{equation} \label{curln_curl_psi}
  \Div \Curl \psi  = 0, \qquad
  \curln \Curl \psi  = - \Delta \psi , \qquad
\DDelta \Curl \psi  = \Curl \Delta \psi .
\end{equation}

The Navier--Stokes equations (NSEs) for the evolution of the (tangential) velocity
vector field $\vu(t,\xb)=(u_\theta(t,\xb),u_\phi(t,\xb))$ on the 2-dimensional rotating
unit sphere $\bS$ under the influence of an external force $f(\xb) = (f_\theta,f_\phi)$ takes the following form
\cite{EbiMar70,Tay92}
\begin{equation}\label{NSE_sphere}  
\partial_t \vu + \nabla_{\vu} \vu - \nu \bL \vu +
  \vomega \times \vu + \frac{1}{\rho} \Grad p = \vf, \quad  \Div \vu = 0,
  \quad \vu(0,\cdot) = \vu_0.
\end{equation}

Let us describe the notations used above in more details.
Firstly, $\nu$ and $\rho$ are two positive constants which can be seen as simplified
physical constants called the viscosity and the density of the fluid.
The word ``rotational'' refers to the Coriolis acceleration $\vomega$ which is normal vector field defined  by
\begin{equation}\label{eqn-2.07}
\vomega=2\Omega\cos\big(\theta(\xb)\big) \xb, \quad \xb \in \bS,
\end{equation}
where $\Omega$ is a given constant. Note that if $\xb = (x_1,x_2,x_3)$ then $\theta(\xb)=\cos^{-1}(x_3)$.

In what follows we will identify the normal vector field $\vomega$ with the corresponding scalar function
$\omega$ defined by
\[
\omega(\xb)=2\Omega\cos\big(\theta(\xb)\big), \qquad \xb \in \bS.
\]

The operator
$\bL$ is given by \cite{Tay92}
\be\label{equ:Lop}
\bL = \DDelta + 2 \mbox{Ric},
\ee
where $\DDelta$ is the Laplace-de Rham operator, see \eqref{deRham},  and Ric denotes the Ricci
tensor of the two-dimensional sphere $\bS$. It is well known that
(see e.g. \cite[page 75]{Wal07})
\begin{equation}\label{equ:RicS}
 \mbox{Ric} = \left[\begin{array}{cc}
                            1  & 0 \\
                            0  & \sin^2\theta
                     \end{array} \right]
                     \end{equation}
We remark that in papers in \cite{CaoRamTit99,Ili91,IliFil89,TemWan93} the authors consider NSEs with  $\bL = \DDelta$
but the analysis in our paper are still valid in that case.
\par

\subsection{Function spaces on the sphere}
In what follows we denote by $dS$ the Lebesgue integration with respect to the surface measure (or the volume measure when $\bS$ is seen as
a Riemannian manifold). In the spherical coordinates we have,
locally, $dS = \sin \theta d\theta d\phi$. For $p\in[1,\infty)$ we will use the notation $L^p=L^p(\bS)$ for the space $L^p\left(\mathbb S^2,\mathbb R\right)$ of $p$-integrable scalar functions
on $\bS$ endowed with the norm
\[\|v\|_{L^p}=\left(\int_{\bS}|v({\vx})|^p\,dS(\vx)\right)^{1/p}.\]
For $p=2$ the corresponding inner product is denoted by
\[\left( v_1,v_2\right)=\left( v_1,v_2\right)_{L^2\left(\bS\right)}=\int_{\bS} v_1v_2\, dS.\]
We will denote by $\mathbb{L}^p=\mathbb{L}^p(\bS)$ the space $L^p\left(\bS,T\bS\right)$ of vector fields $\vv:\bS\to T\bS$
endowed with the norm
\[\|\vv\|_{\mathbb{L}^p}=\left(\int_{\bS}
|\vv({\vx})|^p\,dS(\vx)\right)^{1/p},\]
where, for ${\vx}\in \bS$,  $|\vv(\vx)|$ stands for the length of $\vv(\vx)$ in the tangent space ${T_{\vx}\bS}$. For $p=2$ the corresponding inner product is denoted by
\[\left(\vv_1,\vv_2\right)=\left(\vv_1,\vv_2\right)_{\mathbb{L}^2}=\int_{\bS}\vv_1\cdot\vv_2(S)\,dS.\]
Throughout the paper, the  induced norm on  $\mathbb{L}^2(\bS)$ is
denoted by $\| \cdot \|$ and for other inner product spaces,
say $X$ with inner product $(\cdot,\;\cdot )_{X}$, the associated norm is
denoted by  $\| \cdot \|_X$.

We have the following identities for appropriate scalar and
vector fields~\cite[(2.4)-(2.6)]{Ili91}:
\begin{eqnarray}
(\Grad \psi,\; \vv) &=& - (\psi,\; \Div \vv), \label{ip_identities0} \\
(\Curl \psi,\;\vv)  &=& (\psi,\; \curln \vv), \label{ip_identities1} \\
(\Curl \curln \vw,\; \ovz) &=& (\curln \vw,\; \curln \ovz).
  \label{ip_identities2}
\end{eqnarray}
In \eqref{ip_identities1}, the  $\mathbb{L}^2(\bS)$ inner
product is used on the left hand side and the $L^2(\bS)$ inner product
is used on the right hand side.
We now introduce Sobolev spaces $H^s(\bS)=H^{s,2}(\bS)$ and $\hH^s(\bS)=\hH^{s,2}(\bS)$ of scalar functions and
vector fields on $\bS$ respectively.

Let $\psi$ be a scalar function and let $\vu$ be a vector field on $\bS$, respectively.
For $s\ge 0$ we define
\begin{equation}\label{def:Hs}
  \|\psi\|^2_{H^s(\bS)} = \|\psi\|^2_{L^2(\bS)} + \| (-\Delta)^{s/2} \psi\|^2_{L^2(\bS)},
\end{equation}
and
\begin{equation}\label{def:hHs}
  \|\vu\|^2_{\hH^s(\bS)} = \|\vu\|^2 + \|(-\DDelta)^{s/2} \vu\|^2,
\end{equation}
where $\Delta$ is the Laplace--Beltrami and $\DDelta$ is the Laplace--de Rham operator on the sphere. In particular, for $s=1$,
\begin{align}
\|\vu\|^2_{\hH^1(\bS)}
     &= \|\vu\|^2 + (\vu,-\DDelta \vu)  \nonumber\\
     &= \|\vu\|^2 + \|\Div \vu\|^2 + \| \Curl\del{_{\xb}} \vu\|^2, \label{equ:H1norm}
\end{align}
where we have used formulas
\eqref{deRham},\eqref{ip_identities0}--\eqref{ip_identities2}.

We note that for $k=0,1,2,\ldots$ and $\theta \in (0,1)$ the space $H^{k+\theta}(\bS)$ can
be defined as the interpolation space between $H^k(\bS)$ and $H^{k+1}(\bS)$. We can apply the same procedure for $\hH^{k+\theta}(\bS)$.

One has the following Poincar\'{e} inequality
\cite[Lemma 2]{IliFil89}
\begin{equation}\label{poincare}
  \lambda_1  \|\vu\| \le \|\Div \vu\| + \|\Curl\del{_{\xb}} \vu \|, \quad \vu \in \hH^1(\bS),
\end{equation}
for some positive constant $\lambda_1$.

The space of smooth ($C^\infty$) tangential fields on $\bS$
can be decomposed into three components, one in the space of all
divergence-free fields and the others through the
Hodge decomposition theorem~\cite[Theorem 1.72]{Aub98}:
\begin{equation} \label{Hodge}
  C^\infty(T\bS) =   \calG \oplus  \calV \oplus \calH,
\end{equation}
where
\begin{equation} \label{orth_space}
   \calG = \{\Grad \psi: \psi \in C^\infty(\bS)\},\quad
   \calV = \{\Curl \psi: \psi \in C^\infty(\bS)\},
\end{equation}
while $\calH$ is the finite-dimensional space of harmonic fields,
i.e. $\calH$ contains all the vector fields $\vv$ so that
$\Curl\del{_{\xb}}(\vv) = \Div(\vv) = 0$. Since
the two dimensional sphere is simply connected, $\calH = \{0\}$~\cite[page 80]{Schwarz95}.
We introduce the following spaces
\beas
    H &=& \mbox{ closure of } \calV
          \mbox{ in } \mathbb{L}^2(\bS), \\
    V &=& \mbox{ closure of } \calV
          \mbox{ in } \hH^1(\bS).
\eeas
Since $V$ is densely and continuously embedded
into $H$ and $H$ can be identified
with its dual $H^\prime$, we have the following imbeddings:
\begin{equation}\label{Gelfand triple}
 V \subset H \cong H^\prime \subset V^\prime.
\end{equation}
We say that the spaces $V,H$ and $V^\prime$ form a Gelfand triple.

\subsection{The weak formulation}
We consider the linear Stokes problem
\begin{equation}\label{Stokes}
\nu \Curl {\curln} \vu - 2\nu\mbox{Ric} (\vu) + \Grad p = \vf, \quad \Div \vu = 0.
\end{equation}
By taking the inner product of the first equation of \eqref{Stokes} with $\vv\in V$
and then using \eqref{ip_identities2}, we obtain
\begin{equation}\label{weakStokes}
 \nu (\curln\vu,\curln\vv) - 2\nu(\mbox{Ric\;}\vu,\vv) = (\vf,\vv)  \quad \forall \vv \in V.
\end{equation}
Next, we define a bilinear form $a: V\times V \rightarrow \R$ by
\[
  a(\vu,\vv) :=
  \nu(\curln\vu,{\curln}\del{_{\xb}}\vv) - 2\nu(\mbox{Ric\;}\vu,\vv),\quad \vu,\vv \in V.
\]
In view of \eqref{equ:H1norm} and \eqref{equ:RicS}, the bilinear form
$a$ satisfies
\[
  a(\vu,\vv) \le  \|\vu\|_{\hH^1} \|\vv\|_{\hH^1},
\]
and hence it is continuous on $V$. So by the Riesz Lemma, there exists a unique
operator $\calA:~V \rightarrow V^\prime$,
where $V^\prime$ is the dual of $V$, such that $a(\vu,\vv) = (\calA \vu,\vv)$, for
$\vu, \vv \in V$. Using the Poincar\'{e} inequality \eqref{poincare},
we also have $a(\vu,\vu) \ge \alpha \|\vu\|^2_V$,
with $\alpha = \lambda_1 - 2\nu$, which means $a$
is coercive in $V$ whenever $\lambda_1 > 2\nu$. In practice,
usually one has $\lambda_1 \gg 2\nu$. Hence by the Lax-Milgram
theorem the operator $\calA : V \rightarrow V^\prime$ is an isomorphism. Furthermore,
by using \cite[Theorem 2.2.3]{Tan79}, we conclude that the operator $\calA$ is
positive definite, self-adjoint in $H$ and $\calD(\calA^{1/2}) = V$.



Next we define an operator $\A$ in $H$ as follows:
\begin{equation}\label{defA}
\left\{ \begin{array}{lcl}
            \calD(\A) &:=& \{\vu \in V: \calA\vu \in H \}, \\
            \A \vu    &:=& \calA \vu, \quad \vu \in \calD(\A).
        \end{array} \right.
\end{equation}

Let $\Leray$ be the Leray orthogonal projection from $\mathbb{L}^2(\bS)$ onto $H$.
It can be shown \cite{Gri00} that $\calD(\A) = \hH^2(\bS) \cap V$ and
$\A = -\Leray (\DDelta + 2\mbox{Ric})$,
and
$\A^\ast = \A$. It can also be shown that
$V=\calD(\A^{1/2})$ and
\[
 \|\vu\|^2_V \sim (\A \vu,\vu), \quad \vu \in \calD(\A),
\]
where $A \sim B$ indicates that there are two positive constants $c_1$ and $c_2$
such that $c_1 A \le B \le c_2 A$.


We consider the trilinear form $b$ on  $V \times  V \times V$,
defined as
\begin{equation}\label{trilinear}
b(\vv,\vw,\ovz) = (\nnabla_\vv \vw,  \ovz) =  \int_{{\bS}} \nnabla_\vv \vw \cdot
\ovz~dS, \qquad \vv, \vw, \ovz \in  V.
\end{equation}
Using the following identity
\begin{align}\label{covar_long}
       2 \nnabla_\vw \vv =
       & -{\curln}(\vw\times \vv)+\Grad(\vw\cdot \vv)-\vv\, \Div \vw+ \\
   &\vw \, \Div \vv - \vv\times{\curln} \vw
   - \vw \times {\curln} \vv.
\end{align}
and (\ref{ip_identities1}), for divergence free
tangential vector
fields $\vv, \vw, \ovz$,  the trilinear
form can be written as
\be
b(\vv,\vw,\ovz) =  \frac{1}{2} \int_{\bS}
\left[-\vv \times \vw \cdot \curln {\ovz}  +
         {\curln} \vv \times \vw \cdot {\ovz} - \vv \times {\curln} \vw \cdot {\ovz}\right]~dS.
     \label{short_b}
\ee
Moreover~\cite[Lemma 2.1]{Ili91}
\be \label{skew}
   b(\vv,\vw,\vw)=0, \qquad b(\vv,\ovz,\vw) = -b(\vv,\vw,\ovz)
          \qquad \vv \in V, \vw, \ovz \in \hH^1(\bS).
\ee

The Coriolis operator $\CC_1:\mathbb L^2\left(\bS\right)\to\mathbb L^2\left(\bS\right)$, is defined by the formula
\[
(\CC_1\vv)(\xb) = (2 \Omega \cos \theta(\xb))\xb \times \vv(\xb),\quad \xb \in \bS .
\]
Clearly, $\CC_1$ is linear and bounded in $\mathbb L^2\left(\bS\right)$. In the sequel we will need the operator $\CC=\Leray \CC_1$ which is well defined and bounded in $H$. Furthermore, for $\vu\in H$
\be\label{Cuu}
  (\CC\vu, \vu) = \left(\CC_1{\vu},\Leray \vu\right)=\int_{\bS}2 \Omega \cos \theta(\xb)\big((\xb \times \vu)\cdot\vu(\xb)\big)\,
  d{S}(\xb) = 0.
\ee

Using (\ref{deRham}),  (\ref{ip_identities1}), (\ref{defA}), and (\ref{short_b}),
a \emph{weak solution} of the Navier-Stokes equations (\ref{NSE_sphere}) is a function
$\vu \in L^2([0,T];V)$ with $\vu(0)= \vu_0$ that satisfies the weak form of equation~\eqref{NSE_sphere},
i.e.
\begin{equation} \label{weak_form}
  (\partial_t \vu,\vv) + b(\vu,\vu,\vv) + \nu({\curln} \vu, {\curln} \vv) - 2\nu (\mbox{Ric\;}\vu,\vv)
  + (\CC\vu, \vv)  = (\vf,\vv), \qquad  \vv \in  V.
\end{equation}
This weak formulation can be written in operator equation form on $V^\prime$, the dual of $V$.
Let $\vf \in L^2([0,T];V^\prime)$ and $\vu_0 \in H$. We want to
find a function  $\vu \in L^2([0,T]; V)$, with
$\partial_t \vu \in L^2([0,T];V^\prime)$
such that
\begin{equation}\label{op_form}
   \partial_t \vu + \nu \A\vu + \B(\vu,\vu) + \CC\vu = \vf, \qquad \vu(0) = \vu_0,
\end{equation}
where the bilinear form $\B:V\times V \rightarrow V^\prime$ is defined by
\begin{equation}\label{B}
   (\B(\vu,\vv),\vw) = b(\vu,\vv,\vw) \qquad  \vw \in V.
\end{equation}
With a slight abuse of notation, we also denote $\B(\vu) = \B(\vu,\vu)$.


The following are some fundamental properties of the trilinear form $b$;
see \cite{FenFre05}: there exists a constant $C>0$ such that
\be\label{b_estimate}
|b(\vu,\vv,\vw)| \le C
 \begin{cases}
   \|\vu\|^{1/2} \|\vu\|^{1/2}_V \|\vv\|^{1/2}_V \|\A\vv\|^{1/2} \|\vw\|,
     \quad \vu \in V, \vv\in \calD(\A), \vw \in H,\\
   \|\vu\|^{1/2} \|\A \vu\|^{1/2} \|\vv\|_V \|\vw\|,
   \quad \vu \in \calD(\A), \vv \in V, \vw\in H, \\
   \|\vu\|^{1/2} \|\vu\|^{1/2}_V \|\vv\|_V \|\vw\|^{1/2} \|\vw\|_V^{1/2},
   \quad \vu,\vv,\vw \in V.
 \end{cases}
\ee

We also need the following estimates:
\begin{lemma}\cite[Lemma 2.2]{BrzGolLeG15}
There exists a positive constant $C$ such that
\be\label{b_estimate1}
|b(\vu,\vv,\vw)| \le C \|\vu\| \|\vw\|
 (\|\curln \vv\|_{{\mathbb{L}^{\infty}}} +
 \|\vv\|_{{\mathbb{L}^\infty}}),
\quad \vu \in H, \vv \in V, \vv \in H,
\ee
and
\be\label{b_estimate2}
|b(\vu,\vv,\vw)| \le C \|\vu\| \|\vv\|_V \|\vw\|^{1/2} \|\A\vw\|^{1/2},
\quad \vu \in H, \vv \in V, \vw \in D(\A).
\ee
and
\be\label{b_estimate3}
|b(\vu,\vv,\vw)| \le C \|\vu\|_{\mathbb{L}^4} \| \vv \|_V \|\vw\|_{\mathbb{L}^4},
\quad \vv \in V, \vu,\vw \in \hH^1(\bS).
\ee
\end{lemma}

In view of \eqref{b_estimate3}, $b$ is a bounded trilinear map from
$\mathbb{L}^4(\bS) \times V \times \mathbb{L}^4(\bS)$ to $\R$. Moreover, we have the following
result:

\begin{lemma}\label{lem:bL4}
The trilinear map $b:V \times V \times V \rightarrow \R$ has a unique
extension to a bounded trilinear map from $\mathbb{L}^4(\bS) \cap H \times \mathbb{L}^4(\bS) \times V$
to $\R$.
\end{lemma}

It can be seen from \eqref{b_estimate3} that $b$ is a bounded trilinear map
from $\mathbb{L}^4(\bS)~\times~V~\times~\mathbb{L}^4(\bS)$ to $\R$. It follows that $\B$ maps
$\mathbb{L}^4(\bS) \cap H$ (and so $V$) into $V^\prime$ and
by using the  following inequality from \cite[page 12]{IliFil89}
\be\label{equ:L4}
 \|\vu\|_{\mathbb{L}^4} \le C \|\vu\|^{1/2} \|\vu\|_V^{1/2},
 \quad \vu \in \hH^1(\bS),
\ee
we have
\be\label{equ:Bu L4}
 \|\B(\vu)\|_{V^\prime} \le C_1 \|\vu\|^2_{\mathbb{L}^4} \le  C_2 \|\vu\| \|\vu\|_V
                 \le C_3 \|\vu\|^2_V, \quad \vu \in V.
\ee
\section{The stochastic Navier--Stokes equations on a rotating unit sphere}
\label{sec:sNSEs on S}

\subsection{Preliminaries}

Let us recall that for a real separable Hilbert space  $K$ and a real separable Banach space $X$, a
linear operator $U:K \rightarrow X$ is called $\gamma$-radonifying iff
$\gamma_K \circ U^{-1}$
is $\sigma$-additive. Here $\gamma_K$ is the canonical Gaussian cylindrical measure on $K$. If a linear map  $U: K \rightarrow X$ is $\gamma-$radonifying,
then $\gamma_K\circ U^{-1}$ has a unique extension to a Borel probability measure denoted by
$\nu_U$ on $X$. By $R(K,X)$ we denote the Banach space of $\gamma$-radonifying operators from $K$ to $X$ with the norm
\[
  \|U\|_{R(K,X)} := \left( \int_X |x|^2_X d\nu_U(x) \right)^{1/2}, \quad U \in R(K,X).
\]

From now on we will using freely notation introduced in the former sections.
It follows from \cite[Theorem 2.3]{BrzNee03} that for a self adjoint operator $U\ge cI$ in $H$, where $c>0$, such that $U^{-1}$ is compact,  the operator $U^{-s}:H \rightarrow \mathbb{L}^{p}(\bS)$ is $\gamma$-radonifying iff
\begin{equation}\label{eqrad}
\int_{\bS} \left[ \sum_\ell \lambda^{-2s}_\ell | \ve_\ell(\x) |^2 \right]^{p/2} dS(\xb) < \infty,
\end{equation}
where $\{\ve_\ell\}$ is an orthonormal basis of $H$ corresponding to $U$.  This implies the following result.
\begin{lemma}\label{lem-rad}
Let $\DDelta$ denote the Laplace-de Rham operator on $\mathbb S$. Then the operator
\begin{equation}\label{equ:DeltaRadon}
(-\DDelta)^{-s}:H \rightarrow \mathbb{L}^4(\bS) \quad \mbox{ is }  \gamma-\mbox{ radonifying iff }s > 1/2.
\end{equation}
\end{lemma}
\begin{proof}
Let us recall that all the distinct eigenvalues of $-\DDelta$ are given by a sequence $\lambda_\ell = \ell(\ell+1)$, $\ell=0,1,\ldots$ and the corresponding eigenfunctions are given by
the divergence free vector spherical harmonics $\vY_{\ell,m}$ for $|m| \le \ell$, $\ell \in {\mathbb N}$
\cite[page 216]{VarMosKhe88}. Let us recall also the addition theorem for vector spherical harmonics
\cite[formula (81), page 221]{VarMosKhe88}
\[
  \sum_{|m|\le \ell} |\vY_{\ell,m}(\xb)|^2 = \frac{2\ell+1}{4\pi} P_\ell(1),\;\xb\in \bS,
\]
and the fact that $P_\ell(1)=1$ with $P_\ell$ being the Legendre
polynomial of degree $\ell$. Therefore, \eqref{eqrad} yields
\begin{align}
 \int_{\bS} &\left[ \sum_{\ell=0}^\infty
    (\ell(\ell+1))^{-2s}\sum_{|m|\le \ell} |\vY_{\ell,m}(\xb)|^2 \right]^{4/2} dS(\xb)
  \label{equ:add vec}  \\
 &= \int_{\bS} \left[ \sum_{\ell=0}^\infty
     (\ell(\ell+1))^{-2s} \frac{2\ell+1}{4\pi} P_\ell(1) \right]^2 dS(\xb) < \infty
\nonumber
\end{align}
if and only if $s>\frac{1}{2}$ and the lemma follows.
\end{proof}

Let
\[X=\mathbb{L}^4(\bS) \cap H\]
denote the Banach space endowed with the norm
\[\|x\|_X=\|x\|_H+\|x\|_{\mathbb{L}^4(\bS)}.\]

\begin{remark}\label{rem-rad1} It follows from Lemma \ref{lem-rad} that if  $s > 1/2$ then
 the operator
\begin{equation*}\label{equ:DeltaRadon1}
\A^{-s}:H \rightarrow \mathbb{L}^4(\bS) \cap H \end{equation*}
 is   $\gamma$-radonifying.
\end{remark}

Let us recall, that $X$ is an $M$-type 2 Banach space, see \cite{Brz97} for details.\\
The Stokes operator $-\A$ restricted to $X$ is an infinitesimal generator of  an analytic semigroup. We will consider an operator in $X$ defined by the formula
\[\hat{\A} = \nu \A + \CC,\quad\mathrm{dom}(\hat{\A})=\mathrm{dom}(\A),\]
where $\nu > 0$, and $\CC$ is the Coriolis operator.
For the reader's convenience we recall a result presented in \cite{BrzGolLeG15}.
\begin{proposition}\cite[Proposition 5.3]{BrzGolLeG15}\label{P1}
The operator $\hat{\A}$ with the domain $\mbox{dom}(\hat{\A})=\mbox{dom}(\A)$ generates an analytic $C_0$-semigroup $\left(e^{-t\hat{\A}}\right)$ in $X$. Moreover, there exist constants $\mu>0$, such that for any $\delta\ge 0$ there exists $M_\delta\ge 1$ such that
 \[\|{\hat{\A}}^{\delta} e^{-t\hat{\A}}\|_{\calL(X,X)}\le M_\delta t^{-\delta} e^{-\mu t}\quad t>0.\]
\end{proposition}
Let $E$ denote the completion of $X$
with respect to the image norm $\|\vv\|_E = \|\A^{-\delta} \vv\|_X$, $\vv\in X$. For
$\xi \in (0,1/2)$ we set
\[
  C^\xi_{1/2}(\R,E) := \{\omega \in C(\R,E): \omega(0) = 0,
  \sup_{t,s\in \R} \frac{|\omega(t)-\omega(s)|_E}{|t-s|^\xi(1+|t|+|s|)^{1/2}} < \infty\}.
\]
The space $C^\xi_{1/2}(\R,E)$ equipped with the the norm
\[
 \|\omega\|_{C^\xi_{1/2}(\R,E)} =
 \sup_{t \ne s\in \R} \frac{|\omega(t)-\omega(s)|_E}{|t-s|^\xi(1+|t|+|s|)^{1/2}}
\]
is a nonseparable Banach space. However, the closure of $\{\omega \in C^\infty_0(\R):
\omega(0)=0\}$ in $C^\xi_{1/2}(\R,E)$, denoted by $\Omega(\xi,E)$, is a separable Banach
space.

Let us denote by $C_{1/2}(\R,X)$ the space of all continuous functions
$\omega:\R \rightarrow X$ such that
\[
 \|\omega\|_{C_{1/2}(\R,E)} = \sup_{t\in \R} \frac{|\omega(t)|_E}{1+|t|^{1/2}}<\infty.
\]
The space $C_{1/2}(\R,E)$ endowed with the norm $\|\cdot\|_{C_{1/2}(\R,E)}$
is a nonseparable Banach space.

We denote by $\calF$ the Borel $\sigma$-algebra on $\Omega(\xi,E)$. One can show \cite{Brz96} that
for $\xi \in (0,1/2)$, there exists a Borel probability measure $\bbP$ on $\Omega(\xi,E)$ such that
the canonical process $w_t$, $t\in \R$, defined by
\begin{equation}\label{equ:Wiener}
w_t(\omega) := \omega(t), \quad \omega \in \Omega(\xi,E),
\end{equation}
is a two-sided Wiener process. The Cameron-Martin (or Reproducing Kernel Hilbert space)
of the Gaussian measure $\calL(w_1)$ on $E$ is equal to $K$. For $t \in \R$, let
$\calF_t := \sigma\{w_s: s\le t\}$. Since for each $t \in \R$ the map
$z \circ i_t: E^\ast \rightarrow L^2(\Omega(\xi,E),\calF_t,\bbP)$, where
$i_t: \Omega(\xi, E) \ni \gamma \mapsto \gamma(t) \in E$, satisfies
$\E|z\circ i_t|^2 = t|z|^2_K$, there exists a unique extension of $z\circ i_t$ to a bounded
linear map $W_t: K \rightarrow L^2(\Omega(\xi,E),\calF_t,\bbP)$. Moreover, the family
$(W_t)_{t\in \R}$ is an $H$-cylindrical Wiener process on a filtered probability space
$(\Omega(\xi,E), \mathbb{F},\bbP)$, where $\mathbb{F}=\big(\calF_t)_{t \in \R}$ in the sense of e.g. \cite{Brz+Pesz01}.

\subsection{Ornstein-Uhlenbeck process}\label{sec:OUprocess}

The following is our standing assumption.
\begin{assumption}\label{ass:radon}
Suppose $K \subset H \cap \mathbb{L}^4(\bS)$ is a Hilbert space such that
\be\label{radonify}
  \A^{-\delta} : K \rightarrow H \cap \mathbb{L}^4(\bS) \mbox{ is } \gamma\mbox{-radonifying}
\ee
for some $\delta \in (0,\frac12)$,
\end{assumption}
\begin{remark}\label{rem-radon} It follows from Remark \ref{rem-rad1} that if $K \subset D\big(\A^{s})$ with $s>0$, then Assumption \ref{ass:radon} is satisfied. See also Remark 6.1 in \cite{BrzLi06} and Remark 5.2 in \cite{BrzGolLeG15}.
\end{remark}

On the space $\Omega(\xi,E)$ we consider a flow $\vartheta = (\vartheta_t)_{t\in \R}$ defined by
\[
  \vartheta_t \omega(\cdot) = \omega(\cdot + t) - \omega(t), \quad \omega\in \Omega(\xi,E), \;\;t\in \R.
\]

For $\xi \in (\delta,1/2)$ and $\tilde{\omega} \in C^\xi_{1/2}(\R,X)$ we define
\be\label{zhat}
 \hat{z}(t) = \hat{z}(\hat{\A};\tilde{\omega})(t) =
 \int_{-\infty}^t {\hat{\A}}^{1+\delta} e^{-(t-r)\hat{\A}} (\tilde{\omega}(t)-\tilde{\omega}(r))dr, \quad t \in \R.
\ee
By Proposition~\ref{P1}, for each $\delta > 0 $ there exists $C = C(\delta)>0$ such that
\be\label{semigroup}
 \|{\hat{\A}}^{\delta} e^{-t\hat{\A}} \|_{\calL(X,X)} \le C t^{-\delta} e^{-\mu t}, \quad t  \ge 0.
\ee
 This was an assumption in \cite[Proposition 6.2]{BrzLi06}.
Rewriting that proposition in a slightly more general form we have
\begin{proposition}\label{zhat well def}
For any $\alpha\ge 0$, the operator $-(\hat{\A}+\alpha I)$ is a generator of an analytic semigroup
$\{e^{-t(\hat{\A}+\alpha I)}\}_{t \ge 0}$ in $X$ such that
\[\|{\hat{\A}}^{\delta} e^{-t(\hat{\A}+\alpha I)} \|_{\calL(X,X)} \le C t^{-\delta} e^{-(\mu+\alpha) t}, \quad t  \ge 0.\]
If $t \in \R$, then $\hat{z}(t)$ defined in \eqref{zhat} is a well-defined element of $X$ and for each $t \in \R$
the mapping $\tilde{\omega} \mapsto \hat{z}(t)$ is continuous from $C^\xi_{1/2}(\R,X)$ to $X$.
Moreover, the map $\hat{z} : C^\xi_{1/2}(\R,X) \rightarrow C_{1/2}(\R,X)$ is well defined,
linear and bounded. In particular, there exists a constant $C < \infty$ such that for
any $\tilde{\omega} \in C^{\xi}_{1/2}(\R,X)$
\be\label{equ:zhatgrow}
  | \hat{z}(\tilde{\omega})(t) | \le C (1+ |t|^{1/2})\|\tilde{\omega}\|_{C^{1/2}(\R,X)}.
\ee
\end{proposition}
The following results for the operator $\hat{\A}$ follow from Corollary 6.4, Theorem 6.6 and Corollary 6.8 in
from \cite{BrzLi06}, respectively.
\begin{corollary}
For all $-\infty < a <b<\infty$
and $t \in \R$, for $\tilde{\omega} \in C^\xi_{1/2}(\R,X)$ the map
\[
 \tilde{\omega} \mapsto (\hat{z}(\tilde{\omega})(t),\hat{z}(\tilde{\omega}))
         \in X \times L^4(a,b;X)
\]
is continuous. Moreover, the above result is valid with the space $C^\xi_{1/2}(\R,X)$
being replaced by $\Omega(\xi,X)$.
\end{corollary}

\begin{theorem}\label{thm:vartheta}
For any
$\omega \in C^{\xi}_{1/2}(\R,X)$,
\[
 \hat{z}(\vartheta_s \omega(t)) = \hat{z}(\omega)(t+s), \quad t,s \in \R.
\]
In particular, for any $\omega \in \Omega$ and all $t,s \in \R$,
$\hat{z}(\vartheta_s \omega)(0) = \hat{z}(\omega)(s)$.
\end{theorem}

For $\xi \in C_{1/2}(\R,X)$ we put
\[
  (\tau_s \zeta)(t) = \zeta(t+s), \quad t,s \in \R.
\]
Thus, $\tau_s$ is a linear a bounded map from $C_{1/2}(\R,X)$ into itself. Moreover,
the family $(\tau_s)_{s\in \R}$ is a $C_0$ group on $C_{1/2}(\R,X)$.

Using this notation Theorem~\ref{thm:vartheta} can be rewritten in the following way.
\begin{corollary}\label{cor:flow}
For $s\in \R$, $\tau_s \circ \hat{z} = \hat{z} \circ \vartheta_s$, i.e.
\[
 \tau_s(\hat{z}(\omega)) =
 \hat{z}(\vartheta_s(\omega)), \quad \omega \in C^\xi_{1/2}(\R,X).
\]
\end{corollary}

We define
\[
\ovz_\alpha(\omega) := \hat{z}(\hat{\A} + \alpha I;
          (\hat{\A}+\alpha I)^{-\delta}\omega) \in C_{1/2}(\R,X),
\]
i.e. for any $t\ge 0$,
\begin{eqnarray}
  \ovz_\alpha(\omega)(t) &:=&
   \int_{-\infty}^t (\hat{\A}+\alpha I)^{1+\delta} e^{-(t-r)(\hat{\A} + \alpha I)} \label{zalpha} \\
   && [ (\hat{\A} +\alpha I)^{-\delta} \omega(t) - (\hat{\A} +\alpha I)^{-\delta} \omega(r)   ] dr
   \nonumber
\end{eqnarray}

For $\omega \in C_0^\infty(\R)$ with $\omega(0)=0$, by the fundamental theorem of calculus, we obtain
\begin{eqnarray*}
\frac{d \ovz_\alpha(t)}{dt} &=& -(\hat{\A}+ \alpha I) \int_{-\infty}^t
         (\hat{\A}+\alpha I)^{1+\delta} e^{-(t-r)(\hat{\A} + \alpha I)} \\
&& [ (\hat{\A}+\alpha I)^{-\delta} \omega(t) - (\hat{\A}+\alpha I)^{-\delta} \omega(r)   ] dr
  + \dot{\omega}(t),
\end{eqnarray*}
where $\dot{\omega}(t) = d\omega(t)/dt$.
Hence $\ovz_\alpha(t)$ is the solution of the following equation
\be\label{OUequ}
 \frac{d\ovz_\alpha(t)}{dt} + (\hat{\A}+ \alpha I) \ovz_\alpha = \dot{\omega}(t), \quad t \in \R.
\ee

It follows from Theorem~\ref{thm:vartheta} that
\begin{equation}\label{equ:zalpha}
\ovz_\alpha(\vartheta_s \omega)(t) = \ovz_\alpha(\omega)(t+s),
\quad \omega \in C^{\xi}_{1/2}(\R,X), \; t,s \in \R.
\end{equation}

We can view the formula \eqref{zalpha} as a definition of a process $\ovz_\alpha(t)$, $t \in \R$,
on the probability space $(\Omega(\xi,E),\calF,\bbP)$. Equation \eqref{OUequ} suggests that this
process is an Ornstein-Uhlenbeck process.

\begin{proposition}
\label{pro:stationary}
The process $\ovz_\alpha(t)$, $t\in \R$, is a stationary Ornstein-Uhlenbeck process.
It is the solution of the equation
\[
 d\ovz_\alpha(t) + (\hat{\A}+\alpha I)\ovz_\alpha dt = dw(t), \quad t\in \R,
\]
i.e. for all $t\in \R$, a.s.
\begin{equation}\label{equ:integral}
  \ovz_\alpha(t) = \int_{-\infty}^t e^{-(t-s)(\hat{\A}+\alpha I)} dw(s),
\end{equation}
where the integral is the It\^{o} integral on the $M$-type 2 Banach space $X$ in
the sense of \cite{Brz97}.

In particular, for some constant $C$ depending on $X$,
\[
\E\|\ovz_\alpha(t)\|^4_X
\leq C
\big(\int_0^\infty e^{-2\alpha s} \|e^{-s\hat{\A}}\|^2_{R(K,X)} ds\big)^2.
\]
Moreover, $\E\|\ovz_\alpha(t)\|_X^4$ tends to $0$
as $\alpha \rightarrow \infty$.
\end{proposition}

\begin{proof}
Stationarity of the process $\ovz_\alpha$ follows from equation \eqref{equ:zalpha}.
The equality \eqref{equ:integral} follows by finite-dimensional approximation. 

By the Burkholder inequality, see \cite{Brz97} and \cite{Ondr2004},  we have
\begin{align}
 \E\|\ovz_\alpha(t)\|^4_X &=
  \E \left\| \int_{-\infty}^t e^{-(\hat{\A}+\alpha I)(t-s)} dw(s)\right\|^4_X \nonumber \\
       &\le C \big( \int_{-\infty}^t
\| e^{-(\hat{\A}+ \alpha I)(t-s)}\|^2_{R(K,X)} ds \big)^2
        \label{equ:Eineq1}\\
       &\leq C
\big(\int_0^\infty e^{-2\alpha s} \|e^{-s\hat{\A}}\|^2_{R(K,X)} ds\big)^2.
       \label{equ:Eineq2}
\end{align}

Using \cite[Proposition~5.3]{BrzGolLeG15}
with $\hat{\A} = -\DDelta$, $V=-2\nu \mbox{Ric}+\CC$ and
observation \eqref{equ:DeltaRadon}, we conclude that
\begin{equation}\label{equ:BBB}
\int_0^\infty\|e^{-s\hat{\A}}\|^{2}_{R(K,X)} ds < \infty.
\end{equation}

Hence, we conclude that the last integral \eqref{equ:Eineq2} is finite.
Finally, the last statement follows from \eqref{equ:Eineq2} by applying the Lebesgue
Dominated Convergence Theorem.
\end{proof}

By Proposition~\ref{pro:stationary}, $\ovz_\alpha(t)$, $t\in \R$, is a stationary
and ergodic $X$-valued process, hence by the Strong Law for Large Numbers
(see Da Prato and Zabczyk \cite{PraZab96} for a similar argument),
\begin{equation}\label{equ:Ezalpha}
 \lim_{t \rightarrow \infty} \frac{1}{t} \int_{-t}^0
 \|\ovz_\alpha(s)\|^4_X ds
  = \E \|\ovz_\alpha(0)\|^4_X, \quad \bbP\mbox{-a.s. on } C^{\xi}_{1/2}(\R,X).
\end{equation}

It also follows from Proposition~\ref{pro:stationary} that we can find
$\alpha_0$ such that for all $\alpha \ge \alpha_0$,
\begin{equation}\label{equ:Ezalpha1}
\E \|\ovz_\alpha(0)\|^4_X < \frac{8 \nu^4 \lambda_1}{27C^4},
\end{equation}
where $\lambda_1$ is the constant appearing in the Poincar\'{e} inequality
\eqref{poincare} and $C>0$ is a certain universal constant.

By adding a white noise term to \eqref{NSE_sphere} the stochastic NSEs on the sphere
is
\[
\partial_t \vu + \nabla_{\vu} \vu - \nu \bL \vu + \vomega\times \vu +\Grad p = \vf + n(\x,t),
\quad \Div \vu = 0, \quad \vu(\x,0) = \vu_0,
\]
where we assume that $\vu_0 \in H$, $\vf \in V^\prime$ and $n(t,x)$ is a Gaussian random
field which is a white noise in time. In the same way as in the deterministic case we apply
the operator of projection onto the space of divergence
free fields and reformulate the above equation as an It\^{o} type equation
\begin{equation}\label{sNSEs}
 d\vu(t) + \A \vu(t) dt + \B(\vu(t),\vu(t)) dt + \CC \vu = \vf dt + G dW(t),
 \quad \vu(0) = \vu_0.
 \end{equation}
Here $\vf$ is the deterministic forcing term and $\vu_0$ is the initial velocity.
We assume that $W$ is a cylindrical Wiener process  on a certain Hilbert space $K$ defined
on a probability space $(\Omega,{\mathcal F},\bbP)$, see \cite{DaPZab92} and \cite{Brz+Pesz01}.
$G$ is a linear continuous operator from $K$ to $H$.
The space $K$, which is the RKHS of the Wiener process,  determines the spatial smoothness of the noise term, will satisfy
further assumptions to be specified later.

Roughly speaking, a solution to problem \eqref{sNSEs} is a process $\vu(t)$, $t\ge 0$,
which can be represented in the form
\[\vu(t) = \vv(t) + \ovz_\alpha(t),\;\; t\geq 0,\]
where
$\ovz_\alpha(t)$, $t\in \R$, is a stationary Ornstein--Uhlenbeck process with
drift $-\nu\A - \CC -\alpha I$, i.e. a stationary solution of
\be\label{Orn Uhl}
 d\ovz_\alpha + (\nu\A+\CC +\alpha)\ovz_\alpha dt = G dW(t), \quad t \in \R,
\ee
and $\vv(t)$, $t \ge 0$, is the solution to the following problem
(with $\vv_0 = \vu_0 - \ovz_\alpha(0)$):
\begin{eqnarray}
 \partial_t \vv &=& -\nu \A\vv - \B(\vv+\ovz_\alpha,\vv+\ovz_\alpha)
 - \CC \vv + \alpha \ovz_\alpha + \vf,
  \label{equ:ode} \\
 \quad \vv(0) &=& \vv_0. \label{equ:v0}
\end{eqnarray}
\begin{definition}
Suppose that
$\ovz \in L_{\loc}^4([0,\infty);\mathbb{L}^4(\bS))$, $\vf \in V^\prime$
and $\vv_0 \in H$. A vector field
$\vv \in C([0,\infty);H) \cap L^2_{\loc}([0,\infty);V^\prime) \cap L^4_{\loc}([0,\infty);\mathbb{L}^4(\bS))$
is a solution to problem \eqref{equ:ode}--\eqref{equ:v0} if and only if
$\vv(0) = \vv_0$ and \eqref{equ:ode} holds in the weak sense, i.e. for any $\phi \in V$,
\be\label{def:soln}
\partial_t (\vv,\phi) = -\nu(\vv,\A\phi) - b(\vv+\ovz,\vv+\ovz,\phi)
    - (\CC \vv,\phi) + (\alpha \ovz + \vf, \phi).
\ee
\end{definition}

We remark that for \eqref{def:soln} to make sense, it is sufficient to
assume that $\vv \in L^2(0,T;V) \cap L^\infty(0,T;H)$.

We have proved the following major theorems on the existence and uniqueness
of the solution of $\eqref{equ:ode}-\eqref{equ:v0}$ in \cite{BrzGolLeG15}.

\begin{theorem}\cite[Theorem 3.1]{BrzGolLeG15}
\label{thm:existence}
Assume that $\alpha \ge 0$, $\ovz \in L^4_{\loc}([0,\infty);\mathbb{L}^4(\bS))
\cap L^2_{\loc}([0,\infty);V^\prime)$, $\vv_0~\in H$ and~$\vf\in V^\prime$. Then then there
exists a unique solution $\vv$ of problem $\eqref{equ:ode}-\eqref{equ:v0}$.
\end{theorem}

\begin{theorem}\cite[Theorem 3.2]{BrzGolLeG15}
\label{thm:limit}
Assume that $T>0$ is fixed. If $\vu_{0n} \rightarrow \vu_0$ in $H$,
\[
\ovz_n \rightarrow \ovz \mbox{ in } L^4([0,T];\mathbb{L}^4(\bS))\cap L^2(0,T;V^\prime), \quad
\vf_n \rightarrow \vf \mbox{ in } L^2(0,T;V^\prime).
\]
then
\[
 \vv(\cdot, \ovz_n,\vf_n, \vu_{0n}) \rightarrow \vv(\cdot,\ovz,\vf, \vu_0) \mbox{ in } C([0,T];H) \cap L^2(0,T;V),
\]
where
 $\vv(t,\ovz,\vf ,\vu_0)$ is the solution of problem $\eqref{equ:ode}-\eqref{equ:v0}$
and  $\vv(t,\ovz_n,\vf_n,\vu_{0n})$ is the solution of problem $\eqref{equ:ode}-\eqref{equ:v0}$ with
$\ovz, \vf, \vu_0$ being replaced by $\ovz_n, \vf_n, \vu_{0n}$.
In particular, $\vv(T,\ovz_n,\vu_{0n}) \rightarrow \vv(T,\ovz,\vu_0)$ in $H$.\\
\end{theorem}
%

\section{Attractors for random dynamical systems generated by the stochastic NSEs on the sphere}
\label{sec:attractor}
\subsection{Preliminaries}
A measurable dynamical system (DS) is a triple $$\frakT = (\Omega,\calF,\vartheta),$$ where
$(\Omega,\calF)$ is a measurable space and
$\vartheta:\R \times \Omega \ni (t,\omega) \mapsto \vartheta_t \omega \in \Omega$ is
a measurable map such that for all $t,s \in \R$, $\vartheta_{t+s} = \vartheta_t \circ \vartheta_s$.
A metric DS is a quadruple $$\frakT = (\Omega,\calF,\bbP,\vartheta),$$ where $(\Omega,\calF,\bbP)$
is a probability space and $(\Omega,\calF,\vartheta)$ is a measurable DS such that for each
$t \in \R$, $\vartheta_t : \Omega \rightarrow \Omega$ preserves $\bbP$.

Denote by $\Omega_\alpha(\xi,E)$ the set of those $\omega \in \Omega(\xi,E)$ for which the equality \eqref{equ:Ezalpha} holds true. It follows from Corollary~\ref{cor:flow} that this set is invariant with respect to the flow $\vartheta$, i.e. for all $\alpha \ge 0$ and
all $t \in \R$, $\vartheta_t(\Omega_\alpha(\xi,E)) \subset \Omega_\alpha(\xi,E)$. Therefore,
the same is true for a set
\[
 \hat{\Omega}(\xi,E)= \bigcap_{n=0}^\infty \Omega_n(\xi,E).
\]
It follows that as a model for a metric dynamical system we can take either the
quadruple $(\Omega(\xi,E),\calF,\bbP,\vartheta)$ or the quadruple
$(\hat{\Omega}(\xi,E),\hat{\calF},\hat{\bbP},\hat{\vartheta})$, where
$\hat{\calF}$,$\hat{\bbP}$, and $\hat{\vartheta}$ are respectively the natural
restrictions of $\calF,\bbP$ and $\vartheta$ to $\hat{\Omega}(\xi,E)$.

\begin{proposition}\label{prop:metricDS}
The quadruple $(\hat{\Omega}(\xi,E),\hat{\calF},\hat{\bbP},\hat{\vartheta})$ is a metric DS. For each $\omega \in \hat{\Omega}(\xi,E)$ the limit in \eqref{equ:Ezalpha} exists.
\end{proposition}

Suppose also that $(X,d)$ is a Polish space (i.e. complete separable metric space) and $\calB$
is its Borel $\sigma-$field. Let $\R^{+} = [0,\infty)$.

\begin{definition}
Given a metric DS $\frakT$ and a Polish space $X$, a map
$\varphi:\R^{+}\times \Omega \times X (t,\omega,x) \mapsto \varphi(t,\omega)x \in X$ is called
a measurable random dynamical system (RDS) (on X over $\frakT$) iff
\begin{itemize}
\item[(i)] $\varphi$ is $(\calB(\R^+) \otimes \calF \otimes \calB, \calB)$-measurable.
\item[(ii)] $\varphi(t+s,\omega) = \varphi(t,\vartheta_s \omega) \circ \varphi(s,\omega)$ for
all $s,t \in \R^{+}$ and $\varphi(0,\omega) =id$, for all $\omega \in \Omega$. (Cocycle property)
\end{itemize}
\end{definition}

An RDS $\varphi$ is said to be continuous or differentiable iff for all $(t,\omega) \in \R^{+} \times \Omega$,
$\varphi(t,\cdot,\omega):X \rightarrow X$ is continuous or differentiable, respectively. Similarly,
an RDS $\varphi$ is said to be time continuous iff for all $\omega \in \Omega$ and for all $x\in X$,
$\varphi(\cdot,x,\omega): \R^{+} \rightarrow X$ is continuous.

For two nonempty sets $A,B \subset X$, we put
\[
d(A,B) = \sup_{x\in A} d(x,B) \quad \mbox{ and } \quad
\rho(A,B) = \max\{d(A,B),d(B,A)\}.
\]
In fact, $\rho$ restricted to the family $\mathfrak{CB}$ of all
nonempty closed subsets on $X$ is a metric, and it is called the
Hausdorff metric. From now on, let $\calX$ be the $\sigma$-field
on $\mathfrak{CB}$ generated by open sets with respect to the Hausdorff metric $\rho$; see \cite{CasVal77}.

A set-valued map $C: \Omega \rightarrow \mathfrak{CB}$ is said to be measurable
iff $C$ is $(\calF,\calX)$-measurable. Such a map is often called a {\em closed random set}.

For a given closed random set $B$, the $\Omega$-{\em limit set} of $B$ is defined to be the set
\be\label{def:OmegaLimit}
 \Omega(B,\omega) = \Omega_B(\omega) =
 \bigcap_{T \ge 0} \overline{\bigcup_{t \ge T} \varphi(t,\vartheta_{-t}\omega)B(\vartheta_{-t} \omega)}.
\ee
\begin{definition}
A closed random set $K(\omega)$ is said to (a) attract, (b) absorb, (c) $\rho$-attract another
closed random set $B(\omega)$ iff for all $\omega \in \Omega$, respectively,
\begin{itemize}
\item[(a)]
  $\lim_{t \rightarrow \infty} d(\varphi(t,\vartheta_{-t}\omega) B(\vartheta_{-t} \omega), K(\omega)) = 0$;
\item[(b)] there exists a time $t_B(\omega)$  such that
\[
   \varphi(t,\vartheta_{-t}\omega) B(\vartheta_{-t}\omega) \subset K(\omega)  \mbox{ for all } t\ge t_B(\omega).
\]
\item[(c)]
\[
 \lim_{t\rightarrow \infty} \rho( \varphi(t,\vartheta_{-t}\omega) B(\vartheta_{-t}\omega), K(\omega)) = 0.
\]
\end{itemize}
\end{definition}

We denote by $\calF^u$ the $\sigma-$algebra of universally measurable sets associated to
the measurable space $(\Omega,\calF)$. As far as we are aware, the following definition appeared for the first time as Definition 3.4 in the fundamental work by Flandoli and Schmalfuss \cite{Flandoli-Schmalfuss_1996}, see also \cite{Brzetal10}.

\begin{definition}
\label{randomattractor} A random set $A:\Omega\rightarrow\mathfrak{CB}(X)$  is
a \textbf{random  $\mathfrak{D}$-attractor} iff
\begin{trivlist}
\item[(i)] $A$ is a
compact random set,
\item[(ii)] ${A}$ is  $\varphi$-invariant, i.e., $\bbP$-a.s.
$$
\varphi (t,\omega)A(\omega)=A(\vartheta _t\omega)
$$
\item[(iii)]   $A$  is
 $\mathfrak{D}$-attracting, in the sense that, for all $D\in \mathfrak{D}$ it holds
 $$
   \lim_{t\to \infty}d(\varphi (t, \vartheta _{-t}\omega)D(\vartheta _{-t}\omega), A(\omega))=0.
 $$
\end{trivlist}
\end{definition}

\begin{definition}
We say that an RDS $\vartheta$-cocycle $\varphi$ defined on a separable Banach space $X$ is
$\frakD$-asymptotically compact iff for each $D \in \frakD$,
for every $\omega \in \Omega$, for any positive sequence $(t_n)$ such that $t_n \rightarrow \infty$ and
for any sequence $\{x_n\}$ such that
\[
  x_n \in D(\vartheta_{-t_n} \omega), \quad \mbox{ for all } n \in \nN,
\]
the set $\{\varphi(t_n,\vartheta_{-t_n} \omega)x_n: n \in \nN \}$
is relatively compact in $X$.
\end{definition}

Now we need to state a result on the existence of a random $\frakD$-attractor, see Theorem 2.8 in \cite{Brzetal10} and references therein.
\begin{theorem}
\label{Teorema1} Assume that
$\mathfrak{T}=\left(\Omega,\mathcal{F}, \mathbb{P},
\vartheta\right)$ is a metric DS, $\mathrm{X}$ is a Polish  space,
$\mathfrak{D}$ is a nonempty class of closed and bounded  random sets on $X$ and
$\varphi$ is a continuous, $\mathfrak{D}$-asymptotically compact
RDS on $\mathrm{X}$ (over $\mathfrak{T}$). Assume that there exists a
$\mathfrak{D}$-absorbing closed and bounded  random set $ {B}$ on $X$, i.e. for any  given $D\in \mathfrak{D}$ there exists $t(D)\geq0$ such that
$\varphi (t,\vartheta _{t}\omega)D(\vartheta _{-t}\omega)\subset B(\omega)$ for all $t\geq t(D).$
Then, there exists $\mathfrak{D}$-attractor $A$ given  by
\begin{equation}
A(\mathbb{\omega})=\Omega_{B}(\mathbb{\omega}),\quad\mathbb{\omega}\in\Omega, \label{A}
\end{equation}
with
$$
\Omega_B(\omega)=\bigcap_{T\geq0}\overline{\bigcup_{t\geq
T}\varphi(t, \vartheta_{-t} \omega,B(\vartheta_{-t}
\omega)}),\quad\omega\in\Omega.
$$
which is  $\mathcal{F}^u$-measurable\footnote{By $\mathcal{F}^u$ we understand the $\sigma$-algebra of universally measurable sets associated to
the measurable space $(\Omega,\mathcal{F})$, see the monograph \cite{Crauel_1995} by Crauel.
}.
\end{theorem}

\begin{remark}\label{RACDFremark}
{\rm
If $\frakD$ contains every bounded and closed nonempty deterministic subsets of $X$, then
as a consequence of this theorem, of Theorem 2.1 in \cite{CraDebFla95}, and of Corollary 5.8
in \cite{Cra99} we obtain that the random attractor $A$ is given by
\begin{equation}\label{equ:2_11}
   A(\omega) = \overline{\bigcup_{C \subset X} \Omega_C(\omega)} \quad \bbP - \mbox{ a.s },
\end{equation}
where the union in \eqref{equ:2_11} is made for all bounded and closed nonempty deterministic
subsets $C$ of $X$.
}\end{remark}
\subsection{Random dynamical systems generated by the NSEs}
We fix $\delta<1/2$ and $\xi \in (\delta,1/2)$ and put $\Omega = \Omega(\xi,E)$.
Then we define a map $\varphi = \varphi_\alpha : \R_{+} \times \Omega \times H \rightarrow H$ by
\begin{equation}\label{def:varphi}
\varphi: \R_{+} \times \Omega \times H  \ni (t,\omega,\x) \mapsto \vv(t,\ovz(\omega),\x-\ovz(\omega)(0)) + \ovz(\omega)(t) \in H,
\end{equation}
where $\vv(t,\omega,\vv_0)=\ovz_\alpha(t,\omega,\vv_0)$ is the solution to problem (\ref{equ:ode}-\ref{equ:v0}). \delb{For simplicity of notation we put $\ovz= \ovz_\alpha$.}
Because $\ovz(\omega) \in C_{1/2}(\R,X)$, $\ovz(\omega)(0)$ is a well-defined element of $H$
and hence $\varphi$ is well defined.  It can be shown that $(\varphi,\vartheta)$ is a
random dynamical system (\cite[Theorem 6.1]{BrzGolLeG15}).


Suppose that Assumption~\ref{ass:radon} is satisfied. If $u_s \in H$, $s\in \R$, $f\in V^\prime$
and $W_t$, $t \in \R$ is a two-sided Wiener process introduced after \eqref{equ:Wiener}
such that the Cameron-Martin (or Reproducing Kernel Hilbert) space of the Gaussian measure
$\calL(w_1)$ is equal to $K$. A process $\vu(t)$, $t\ge 0$, with trajectories in
$C([s,\infty);H)\cap L^2_{\loc}([s,\infty);V)\cap L^2_{\loc}([s,\infty);\mathbb{L}^4(\bS))$ is a solution
to problem \eqref{sNSEs} iff $\vu(s) = \vu_s$ and for any $\vphi \in V$, $t > s$,
\be\label{def:sol sNSE}
\begin{aligned}
(\vu(t),\vphi) &= (\vu(s),\vphi) - \nu \int_s^t (\A\vu(r),\vphi)dr - \int_s^t b(\vu(r),\vu(r),\vphi)dr \\
               & - \int_s^t (\CC\vu(r),\vphi) dr + \int_s^t (\vf,\vphi)dr
               + \int_s^t \langle \vphi, dW_r \rangle.
\end{aligned}
\ee

In the framework as above, suppose that $\vu(t) = \ovz_\alpha(t) + \vv_\alpha(t)$, $t \ge s$,
where $\vv_\alpha$ is the unique solution to problem \eqref{equ:ode}--\eqref{equ:v0} with
initial data $\vu_0 - \ovz_\alpha(s)$ at time $s$. If the process $\vu(t)$, $t \ge s$, has
trajectories in $C([s,\infty);H)\cap L^2_{\loc}([s,\infty);V)\cap L^2_{\loc}([s,\infty);\mathbb{L}^4(\bS))$, then
it is a solution to problem \eqref{sNSEs}. Vice-versa, if a process $\vu(t)$, $t \ge s$,
with trajectories in   $C([s,\infty);H)\cap L^2_{\loc}([s,\infty);V)\cap L^2_{\loc}([s,\infty);\mathbb{L}^4(\bS))$
is a solution to problem \eqref{sNSEs}, then for any $\alpha \ge 0$, a process $\vv_\alpha(t)$, $t \ge s$,
defined by $\ovz_\alpha(t) = \vu(t) - \vv_\alpha(t)$, $t\ge s$, is a solution to \eqref{equ:ode} on $[s,\infty)$.

Our previous results yield the existence and the uniqueness of solutions to problem \eqref{sNSEs} as well
as its continuous dependence on the data (in particular on the initial value $\vu_0$ and the force $\vf$).
Moreover, if we define, for $\x \in H$, $\omega \in \Omega$, and $t \ge s$,
\be\label{def:uv}
\vu(t,s; \omega,\vu_0) := \varphi(t-s; \vartheta_s \omega)\vu_0 =
 \vv (t,s; \omega,\vu_0 -\ovz(s)) + \ovz(t),
\ee
then for each $s\in \R$ and each $\vu_0 \in H$, the process $\vu(t)$, $t \ge s$, is a solution to problem
\eqref{sNSEs}.


We have the Poincar\'{e} inequalities
\be\label{equ:Poincare}
 \begin{aligned}
   \|\vu\|_V^2 &\ge \lambda_1 \|\vu\|^2, \quad \mbox{ for all } \vu \in V, \\
   \|\A\vu\|^2 &\ge \lambda_1 \|\vu\|^2, \quad \mbox{ for all } \vu \in
         \calD(A)\cap V.
 \end{aligned}
\ee

For any $\vu,\vv \in V$, we define a new scalar product $[\cdot,\cdot]:V \times V \rightarrow \R$ by the formula
$[\vu,\vv] = \nu(\vu,\vv)_V - \nu \frac{\lambda_1}{2}(\vu,\vv)$. Clearly,
$[\cdot,\cdot]$ is bilinear and symmetric. From \eqref{poincare}, we can prove
that $[\cdot,\cdot]$ define an inner product in $V$ with the norm
$[\cdot]=[\cdot,\cdot]^{1/2}$, which is equivalent to the norm $\|\cdot\|_V$.

The following lemma is given in \cite[Lemma 6.5]{BrzGolLeG15}, however
the proof there is not quite
correct.
The bound on the nonlinear term $b(\vv,\ovz,\vv)$ there was not
treated correctly, hence the power on the stochastic term
$\|\ovz\|_{\mathbb{L}^4}$ was not correctly stated. The error propagated
to the rest of the paper.
We present a corrected proof here.
\begin{lemma}\label{newlemma6_5}
\label{lem:NSEsol}
Suppose that $\vv$ is a solution to problem \eqref{equ:ode} on the time interval $[a,\infty)$ with
$\ovz \in L^4_{\loc}(\R^+, \mathbb{L}^4(\bS)) \cap L^2_{\loc}(\R^+,V^\prime)$ and $\alpha \ge 0$. Denote
$\vg(t) =\alpha \ovz(t) - \B(\ovz(t),\ovz(t))$, $t\in [a,\infty)$. Then, for any $t \ge \tau \ge a$,
\be\label{equ:8_10}
\begin{aligned}
\|\vv(t)\|^2 &\le \|\vv(\tau)\|^2
\exp\left(-\nu \lambda_1(t-\tau)+ \frac{27C^4}{4\nu^3} \int_\tau^t \|\ovz(s)\|^4_{\mathbb{L}^4} ds\right) \\
 & + \frac{3}{\nu} \int_\tau^t ( \|\vg(s)\|^2_{V^\prime} + \|\vf(s)\|_{V^\prime}^2 )
 \exp \left(-\nu \lambda_1(t-\tau)+ \frac{27C^4}{4\nu^3} \int_s^t \|\ovz(\xi)\|^4_{\mathbb{L}^4} d\xi \right)
 ds
\end{aligned}
\ee
\be\label{equ:8_11}
\begin{aligned}
\|\vv(t)\|^2 &= \|\vv(\tau)\|^2 e^{-\nu \lambda_1(t-\tau)}  \\
& + 2 \int_\tau^t e^{-\nu \lambda_1(t-s)}(b(\vv(s),\ovz(s),\vv(t)) + \langle \vg(s),\vv(s) \rangle + \langle\vf,\vv(s)\rangle
    - [\vv(s)]^2) ds
\end{aligned}
\ee
\end{lemma}

\begin{proof}
By \cite[Lemma III.1.2]{Tem79}, we have
$\frac{1}{2}\partial_t \|\vv(t)\|^2 = ( \partial_t \vv(t),\vv(t))$. Hence
\begin{equation}\label{eqn:812}
\begin{aligned}
\frac{1}{2}\frac{d}{dt} \|\vv\|^2 &=
-\nu(\A\vv,\vv) - (\CC\vv,\vv) - (\B(\vv,\vv),\vv) - (B(\ovz,\vv),\vv) \\
& \qquad -(B(\vv,\ovz),\vv) + \langle \vg,\vv\rangle +
\langle\vf,\vv\rangle\\
& = -\nu \|\vv\|^2_V - b(\vv,\ovz,\vv) +
\langle\vg,\vv \rangle + \langle \vf,\vv\rangle.
\end{aligned}
\end{equation}
From \eqref{b_estimate3} and invoking the Young inequality,
we have
\begin{align*}
|b(\vv,\ovz,\vv)| &\le C\|\vv\|_{\mathbb{L}^4} \|\vv\|_V \|\ovz\|_{\mathbb{L}^4} \\
        & \le C \| \vv \|^{1/2} \| \vv \|_V^{3/2} \|\ovz\|_{\mathbb{L}^4} \\
        &\le \frac{\nu}{2}\|\vv\|^2_V +
     \frac{27 C^4}{32\nu^3} \|\vv\|^2 \|\ovz\|^4_{\mathbb{L}^4},
\end{align*}
and
\begin{align*}
|\langle \vg,\vv\rangle+\langle\vf,\vv\rangle| &\le \|\vg\|_{V^\prime} \|\vv\|_V   + \|\vf\|_{V^\prime}\|\vv\|_{V} \\
 &\le \frac{\nu}{3}     \|\vv\|_V^2 + \frac{3}{2\nu} \|g\|^2_{V^\prime} + \frac{3}{2\nu}
 \|\vf\|^2_{V^\prime}.
\end{align*}
Hence from \eqref{eqn:812} and \eqref{equ:Poincare}, we get
\begin{align*}
\frac{d}{dt} \|\vv(t)\|^2 &\le  -\nu\| \vv(t) \|_V^2 +
\frac{27C^4}{4\nu^3}\|\ovz(t)\|^4_{\mathbb{L}^4} \|\vv(t)\|^2 +
  \frac{3}{\nu}\|\vg(t)\|^2_{V^\prime} + \frac{3}{\nu} \|\vf\|^2_{V^\prime} \\
  &\le\left(-\nu\lambda_1 + \frac{27C^4}{16\nu^3}\|\ovz(t)\|^4_{\mathbb{L}^4}\right)
  \|\vv(t)\|^2
  + \frac{3}{\nu}\|\vg(t)\|^2_{V^\prime} + \frac{3}{\nu} \|\vf\|^2_{V^\prime}.
\end{align*}
Next, using the Gronwall Lemma, we arrive at \eqref{equ:8_10}.

By
adding and subtracting $\nu \frac{\lambda_1}{2}\|\vv(t)\|^2$
from \eqref{eqn:812}
we find that
\begin{align}
\frac{d}{dt}\|\vv(t)\|^2 &+ \nu \lambda_1 \|\vv(t)\|^2 + 2[\vv(t)]^2\\
    & = 2b(\vv(t),\ovz(t),\vv(t)) + 2 \langle \vg(t),\vv(t)\rangle + 2\langle\vf(t),\vv(t)\rangle.
\end{align}
Hence \eqref{equ:8_11} follows by the variation of constants formula.
\end{proof}

\begin{lemma} \label{lem_8.5}
Under the above assumptions, for each $\omega \in \Omega(\xi,E)$,
\[
  \lim_{t \rightarrow -\infty} \| \ovz(\omega)(t) \|^2
         \exp \left( \nu \lambda_1 t  +
     \frac{27C^4}{16\nu^3} \int_t^0  \|\ovz(\omega)(s)\|^4_{\mathbb{L}^4} ds \right) = 0.
\]
\end{lemma}
\begin{lemma}\label{lem_8.6}
Under the above assumptions, for each $\omega \in \Omega(\xi,E)$,
\[
\int_{-\infty}^0 [1+ \|\ovz(\omega)(t)\|^2_{\mathbb{L}^4} + \|\ovz(\omega)(t)\|^4_{\mathbb{L}^4}]
          \exp \left(
 \nu \lambda_1 t  +  \frac{27C^4}{16\nu^3}\int_t^0 \|\ovz(\omega)(s)\|^4_{\mathbb{L}^4} ds \right) < \infty.
\]
\end{lemma}
\begin{definition}\label{def classR}
 A function $r:\Omega \rightarrow (0,\infty)$ belongs to the class $\frakR$ if and only
if
\[
 \limsup_{t \rightarrow -\infty} r(\vartheta_{-t} \omega)^2
\exp\left(\nu \lambda_1 t + \frac{27C^4}{16\nu^3}\int_t^0  \|\ovz(\omega)(s)\|^4_{\mathbb{L}^4} ds\right) = 0,
\]
where $C>0$ is the constant appearing in \eqref{equ:Ezalpha1}.

We denote by $\mathfrak{DR}$ the class of all closed and bounded random sets $D$ on
$H$ such that the function
$\omega \mapsto r(D(\omega)) := \sup \{ \|\vx\|_H : \vx \in D(\omega) \}$
belongs to the class $\mathfrak{R}$.
\end{definition}

\begin{proposition}
\label{prop_class_R} Define functions $r_i:\Omega \to (0,\infty)$,
$i=1,2,3,4,5$ by the following formulae, for $\omega\in \Omega$,
\begin{eqnarray*}
r_1^2(\omega)&:=& \|\ovz(\omega)(0)\|_H^2, \\
r_2^2(\omega)&:=& \sup_{s\leq 0}\| \ovz (\omega)(s)\|_H^2
\exp \left(\nu\lambda_1s+\frac{27C^4}{16\nu^3}\int^0_{s}\|\ovz(\omega)(r)\|^4_{\mathbb{L}^4}\,dr\right) \\
r_3^2(\omega)&:=& \int_{-\infty}^0 \| \ovz (\omega)(s)\|^2_H
\exp \left( \nu\lambda_1s+\frac{27C^4}{16\nu^3}\int^0_{s}\|\ovz(\omega)(r)\|^4_{\mathbb{L}^4}\,dr \right)\,ds \\
r_4^2(\omega)&:=& \int_{-\infty}^0 \| \ovz (\omega)(s)\|^4_{\mathbb{L}^4}
\exp \left(\nu\lambda_1s+\frac{27C^4}{16\nu^3}\int^0_{s}\|\ovz(\omega)(r)\|^4_{\mathbb{L}^4}\,dr \right)\,ds\\
r_5^2(\omega)&:=& \int_{-\infty}^0
\exp\left(\nu\lambda_1s+\frac{27C^4}{16\nu^3}\int^0_{s}\|\ovz(\omega)(r)\|^4_{\mathbb{L}^4}\,dr\right)\,ds.
\end{eqnarray*}
Then all these functions belong to the class $\mathfrak{R}$.
\end{proposition}
\begin{proof}
Since by Theorem \ref{thm:vartheta},
$\ovz(\vartheta_{-t}\omega)(s)=\ovz(\omega)(s-t)$, we have
\begin{eqnarray*}
r_2^2(\vartheta_{-t} \omega)&=&
\sup_{s\leq 0}\| \ovz(\vartheta_{-t} \omega)(s)\|^2
\exp \left( \nu\lambda_1s+\frac{27C^4}{16\nu^3}
\int^0_{s}\|\ovz(\vartheta_{-t} \omega) (r)\|^4_{\mathbb{L}^4}\,dr\right) \\
&=& \sup_{s\leq 0}\| \ovz (\omega)(s-t)\|^2
 \exp \left(
        \nu\lambda_1s+\frac{27C^4}{16\nu^3}\int^0_{s}\|
        \ovz(\omega)(r-t)\|^2_{\mathbb{L}^4}\,dr \right) \\
        &=& \sup_{s\leq 0}\| \ovz (\omega)(s-t)\|^2
        \exp \left( \nu\lambda_1(s-t)+\frac{27C^4}{16\nu^3}\int^{-t}_{s-t}\|
        \ovz(\omega)(r)\|^4_{\mathbb{L}^4}\,dr \right) e^{\nu\lambda_1t} \\
        &=& \sup_{\sigma \leq -t}\| \ovz (\omega)(\sigma)\|^2
        \exp\left(\nu\lambda_1\sigma+\frac{27C^4}{16\nu^3}\int^{-t}_{\sigma}\|
        \ovz(\omega)(r)\|^4_{\mathbb{L}^4}\,dr\right) e^{\nu\lambda_1t} \\
        \end{eqnarray*}
Hence, multiplying the above by
        $\exp\left( -\nu \lambda_1 t + \frac{27C^4}{16\nu^3}\int_{-t}^0\|
        \ovz(\omega)(r)\|^4_{\mathbb{L}^4}\,dr \right)$ we get
\begin{align*}
   &     r_2^2(\vartheta_{-t} \omega)
   \exp\left(-\nu\lambda_1t+
        \frac{27C^4}{16\nu^3}\int_{-t}^0\|
        \ovz(\omega)(r)\|^4_{\mathbb{L}^4}\,dr \right) \\
   & \qquad     \leq  \sup_{\sigma \leq -t}\| \ovz (\omega)(\sigma)\|^2
        \exp \left( \nu\lambda_1\sigma+\frac{27C^4}{16\nu^3}\int^{0}_{\sigma}
\|\ovz(\omega)(r)\|^4_{\mathbb{L}^4}\,dr\right).
\end{align*}
This,  together with  Lemma \ref{lem_8.5} concludes the proof in
the case of function $r_2$.  In the case of $r_1$, we have
\begin{align*}
&  r_1^2(\vartheta_{-t} \omega)
\exp\left(-\nu\lambda_1t+
        \frac{27C^4}{16\nu^3}\int_{-t}^0
\| \ovz(\omega)(r)\|^4_{\mathbb{L}^4}\,dr \right) \\
&\qquad =
\| \ovz (\omega)(-t)\|^2
      \exp\left(-\nu\lambda_1 t+\frac{27C^4}{16\nu^3}\int^{0}_{-t}
            \|\ovz(\omega)(r)\|^4_{\mathbb{L}^4}\,dr\right).
\end{align*}
Thus, by Lemma \ref{lem_8.5} we infer that $r_1$ also belongs to
the class $\mathfrak{R}$. The argument in the case of function
$r_3$ is similar but for the sake of the completeness we include it here.
From the first part of the proof we infer that
\begin{align*}
 &       r_3^2(\vartheta_{-t} \omega)
       \exp\left(-\nu\lambda_1t+
        \frac{27C^4}{16\nu^3}\int_{-t}^0\|
        \ovz(\omega)(r)\|^4_{\mathbb{L}^4}\,dr\right) \\
& \qquad \leq
        \int_{-\infty}^{-t}\| \ovz (\omega)(\sigma)\|^2
        \exp\left(\nu\lambda_1\sigma+\frac{27C^4}{16\nu^3}\int^{0}_{\sigma}\|
        \ovz(\omega)(r)\|^4_{\mathbb{L}^4}\,dr\right)\, d\sigma.
\end{align*}
Since  by Lemma \ref{lem_8.6},
$\int_{-\infty}^{0}\| \ovz
        (\omega)(\sigma)\|^2
        \exp \left(\nu\lambda_1\sigma+\frac{27C^4}{16\nu^3}\int^{0}_{\sigma}\|
        \ovz(\omega)(r)\|^4_{\mathbb{L}^4}\,dr\right)\, d\sigma$ is finite, by
        the Lebesgue Monotone Convergence Theorem we conclude that
        $$\int_{-\infty}^{-t}\| \ovz (\omega)(\sigma)\|^2
        \exp\left(\nu\lambda_1\sigma+
          \frac{27C^4}{16\nu^3}\int^{0}_{\sigma}
\|\ovz(\omega)(r)\|^4_{\mathbb{L}^4}\,dr\right)\, d\sigma \to 0 \mbox{ as } t\to\infty.$$
The proof in the other cases is analogous.
\end{proof}

We have the following trivial results.
\begin{proposition}
The class $\mathfrak{R}$ is closed  with respect to sum,
multiplication by a constant and if $r\in\mathfrak{R}$, $0\leq
\bar{r}\leq r$, then $\bar{r}\in\mathfrak{R}$.
The class $\mathfrak{R}$ is closed  with respect to sum,
multiplication by a constant and if $r\in\mathfrak{R}$, $0\leq
\bar{r}\leq r$, then $\bar{r}\in\mathfrak{R}$.
\end{proposition}

Now we are ready to state and prove the main result of this
paper. A result of similar type for the Navier--Stokes equations on
some 2-dimensional unbounded domain has been discussed in
\cite{Brzetal10}.
\begin{theorem}\label{thm-main}
Consider the metric DS
$\mathfrak{T}=\left(\hat{\Omega}(\xi,\mathrm{E}),\hat{\mathcal{F}},
\hat{\mathbb{P}},\hat{\vartheta}\right)$ from Proposition
\ref{prop:metricDS}, and the RDS  $\varphi$ on $H$ over
$\mathfrak{T}$ generated by the stochastic Navier-Stokes equations on the
$2$-dimensional unit sphere with additive noise
\eqref{sNSEs} satisfying Assumption~\ref{ass:radon}.
Then the following properties hold.
 \begin{trivlist}
  \item[(i)] there exists a $\mathfrak{DR}$-absorbing set $B\in\mathfrak{DR}$;
   \item[(ii)] the RDS $\varphi$ is $\mathfrak{DR}$-asymptotically compact;
   \item[(iii)] the family  $A$ of sets defined by $A(\omega)=\Omega_B(\omega)$ for
    all $\omega\in\Omega,$ is the minimal $\mathfrak{DR}$-attractor for $\varphi,$
    is $\hat{\mathcal{F}}$-measurable, and
    \begin{equation}\label{RACDFH}
    A(\omega)=\overline{\bigcup_{C\subset H}\Omega_C(\omega)}\quad\hat{\mathbb{P}}-a
    .s.,
    \end{equation}where the union in \eqref{RACDFH}
    is made for all bounded and closed nonempty deterministic subsets $C$
of $H$.
 \end{trivlist}
  \end{theorem}
\begin{proof}
In view of Theorem \ref{Teorema1} and Remark \ref{RACDFremark},
it is enough to show (i) and (ii). The proof of (ii) will
be done in the next proposition.

\noindent
{\em Proof of (i)}

With a fixed $\omega \in \Omega$,
let $D(\omega)$ be a random  set from the class $\mathfrak{DR}$
with radius $r_D(\omega)$, i.e. $r_D(\omega):=\sup\{|\vx|_H:x\in
D(\omega)\}$.

For given  $s \leq 0$ and $\vx \in
H$, let $\vv$ be the solution of (\ref{equ:ode}) on time
interval $[s, \infty)$ with the initial condition $\vv(s)= \vx-\ovz(s)$.
By applying (\ref{equ:8_10}) with $t=0, \tau =s \leq 0 $,  we get
     \begin{eqnarray}\label{eqn:5.13}
     \| \vv(0)\|^2 &\leq&  2\|\vx\|^2
     \exp\left(\nu\lambda_1s+\frac{27C^4}{16\nu^3}
\int_{s}^0
   \| \ovz(r)\|^4_{\mathbb{L}^4}\, \,dr\right) \\
&&    + 2\| \ovz(s)\|^2
     \exp\left(\nu\lambda_1s+\frac{27C^4}{16\nu^3}
\int_{s}^0
   \| \ovz(r)\|^4_{\mathbb{L}^4}\, \,dr\right) \nonumber\\
&& +\frac{3}{\nu}
\int_{s}^0
  \{\|
     \vg(t)\|^2_{\mathrm{V}^\prime}+\| \vf \|^2_{\mathrm{V}^\prime}
     \}
  \exp\left(\nu\lambda_1t+\frac{27C^4}{16\nu^3}\int_{t}^0\|
     \ovz(r)\|^4_{\mathbb{L}^4}\, dr\right)\,dt.
\end{eqnarray}
Set, for $\omega\in\Omega$,%
\begin{eqnarray}
      r_{11}(\omega)^2= 2 +\sup_{s\leq 0}
  \left\{2\|
       \ovz(s)\|^2
    \exp\left(\nu\lambda_1 s+
     \frac{27C^4}{16\nu^3}
   \int_{s}^0
      \|\ovz(r)\|^4_{\mathbb{L}^4}\, \,dr \right) \right.
      \nonumber\\
      +\left.\frac{3}{\nu}
 \int_{s}^0
 \{\|
      \vg(t)\|^2_{\mathrm{V}^\prime}+\| \vf \| ^2_{\mathrm{V}^\prime}
      \}
\exp\left(\nu\lambda_1 t+
     \frac{27C^4}{16\nu^3}
   \int_{t}^0\|
      \ovz(r)\|^4_{\mathbb{L}^4}\, dr\right)\,dt \right\},
       \label{eqn:r_11}
\end{eqnarray}
and
\begin{equation}
        r_{12}(\omega)=\|\ovz(0)(\omega)\|_{H} \label{eqn:r_12}.
\end{equation}

Using Lemma \ref{lem_8.6} and Proposition \ref{prop_class_R}  we conclude
that both $r_{11}$ and $r_{12}$ belong to $\mathfrak{R}$ and
that $r_{13}:=r_{11}+r_{12}$  belongs to $\mathfrak{R}$ as well.
Therefore, the random set $B$ defined by
$B(\omega):=\{\vu \in H: \| \vu\|  \leq r_{13}(\omega)\}$ belongs to the
family $\mathfrak{DR}$.

Now we will show that $B$ absorbs $D$. Let $\omega\in\Omega$ be
fixed. Since $r_D\in\mathfrak{R}$ there exists $t_D(\omega)\geq
0$, such that
\[
r_D(\vartheta_{-t}\omega)^2
\exp \left( -\nu\lambda_1 t
+\frac{27C^4}{16\nu^3}\int_{-t}^0\|\ovz(\omega)(s)\|^4_{\mathbb{L}^4}\, ds \right)\leq 1,
\mbox{ for } t\geq t_D(\omega).
\]
Thus, if $\vx\in
D(\vartheta_{-t}\omega)$ and $s\geq t_D(\omega)$, then by
\eqref{eqn:5.13},  $\| \vv(0,s;\omega,\vx-\ovz(s))\| \leq
r_{11}(\omega)$. Thus we infer that
\begin{equation*}
\|\vu(0,s;\omega,\vx)\|  \leq \| \vv(0,s;\omega,\vx-\ovz(s))\| +
\|\ovz(0)(\omega)\| \leq r_{13}(\omega).
\end{equation*}
In other words, $\vu(0,s;\omega,\vx)\in B(\omega)$, for all $s\geq
t_D(\omega)$. This proves that $B$ absorbs $D$.
\end{proof}

\begin{proposition}\label{prop:ACprop}
Assume that for each random set $D$ belonging to $\mathfrak{DR}$, there exists a random set
$B$ belonging to $\mathfrak{DR}$ such that $B$ absorbs $D$. Then the RDS $\varphi$ is
$\mathfrak{DR}$-asymptotically compact.
\end{proposition}


The proof of the proposition is adapted from \cite{Brzetal10},
in which a RDS generated by NSEs on some 2-dimensional unbounded domain
was considered. The proposition generalises the asymptotically compactness of the RDS in \cite[Proposition 8.1]{BrzLi06} to the $\mathfrak{DR}$- asymptotically compactness of the RDS.

\noindent
\begin{proof}

Suppose that $B$ is a closed random set from the class $\mathfrak{DR}$ and $K \in \mathfrak{DR}$ is a
close random set which absorbs $B$. We fix $\omega \in \Omega$. Let us take an increasing sequence of positive
numbers $(t_n)_{n=1}^\infty$ such that $t_n \rightarrow \infty$ and an $H$-valued sequence $(\x_n)_n$
such that $\x_n \in B(\vartheta_{-t_n}\omega)$, for all $n \in \nN$.

{\bf Step I.} Reduction. Since $K(\omega)$ absorbs $B$, for $n\in \nN$ sufficiently large,\linebreak
$\varphi(t_n,\vartheta_{-t_n} \omega) B \subset K(\omega)$. Since $K(\omega)$ is closed and bounded,
and hence weakly compact, without loss of generality we may assume that
$\varphi(t_n,\vartheta_{-t_n}\omega) B \subset K(\omega)$ for all $n \in \nN$ and, for some $\vy_0 \in K(\omega)$,
\be\label{def:y0}
  \varphi(t_n,\vartheta_{-t_n} \omega) \x_n \rightarrow \vy_0 \quad \mbox{ weakly in } H.
\ee
Since $\ovz(0) \in H$, we also have
\[
  \varphi(t_n,\vartheta_{-t_n}\omega)\x_n - \ovz(0) \rightarrow \vy_0 - \ovz(0) \quad \mbox{ weakly in } H.
\]
In particular,
\be\label{equ:liminf}
 \|\vy_0 - \ovz(0)\| \le \liminf_{n\rightarrow \infty} \| \varphi(t_n,\vartheta_{-t_n}\omega) \x_n - \ovz(0)\|.
\ee

We claim that it is enough to prove that for some subsequence $\{ n' \} \subset \nN$
\be\label{equ:limsup}
\|\vy_0 - \ovz(0)\| \ge \limsup_{n'\rightarrow \infty} \| \varphi(t_{n'},\vartheta_{-t_{n'}}\omega) \x_{n'} - \ovz(0)\|.
\ee
Indeed, since $H$ is a Hilbert space, \eqref{equ:liminf} in conjunction with \eqref{equ:limsup} imply that
\[
 \varphi(t_n,\vartheta_{-t_n}\omega)\x_n - \ovz(0) \rightarrow \vy_0 - \ovz(0) \quad \mbox{ strongly in } H
\]
which implies that
\[
\varphi(t_n,\vartheta_{-t_n} \omega) \x_n \rightarrow \vy_0 \quad \mbox{ strongly in } H.
\]
Therefore, in order to show that $\{\varphi(t_n,\vartheta_{-t_n}\omega)\x_n\}_n$ is relatively compact in $H$ we need
to prove that \eqref{equ:limsup} holds true.

{\bf Step II.} Construction of a negative trajectory, i.e. a sequence $(\vy_n)_{n=-\infty}^0$ such that
$\vy_n \in K(\vartheta_n \omega)$, $n \in \zZ^{-}$, and $\vy_k = \varphi(k-n,\vartheta_n \omega)\vy_n$, $n<k \le 0$.

Since $K(\vartheta_{-1}\omega)$ absorbs $B$, there exists a constant $N_1(\omega) \in \nN$, such that
\[
 \{ \varphi(-1+t_n,\vartheta_{1-t_n} \vartheta_{-1} \omega)\x_n : n \ge N_1(\omega) \} \subset K(\vartheta_{-1}\omega).
\]
Hence we can find a subsequence $\{n'\} \subset \nN$ and $\vy_{-1} \in K(\vartheta_{-1}\omega)$ such that
\be\label{def:ym1}
 \varphi(-1+t_{n'}, \vartheta_{-t_{n'}}\omega)\x_{n'} \rightarrow \vy_{-1} \mbox{ weakly in } H.
\ee
We observe that the cocycle property, with $t=1$, $s=t_{n'}-1$, and $\omega$ being replaced by $\vartheta_{-t_{n'}}\omega$,
reads as follows:
\[
\varphi(t_{n'},\vartheta_{-t_{n'}}\omega) = \varphi(1,\vartheta_{-1}\omega) \varphi(-1+t_{n'},\vartheta_{t_{n'}}\omega).
\]
Hence, by the last part of Theorem \ref{thm:limit}, from \eqref{def:y0} and \eqref{def:ym1} we infer that
$\varphi(1,\vartheta_{-1}\omega) \vy_{-1} = \vy_0$. By induction, for each $k=1,2,\ldots,$ we can construct a subsequence
$\{n^{(k)} \} \subset \{ n^{(k-1)} \}$ and $\vy_{-k} \in K(\vartheta_{-k}\omega)$, such that
$\varphi(1,\vartheta_{-k}\omega)\vy_{-k} = \vy_{-k+1}$ and
\be\label{def:ymk}
\varphi(-k+ t_{n^{(k)}}, \vartheta_{-t_{n^{(k)}}}\omega)\x_{n^{(k)}} \rightarrow \vy_{-k} \mbox{ weakly in } H,
 \mbox{ as } n^{(k)} \rightarrow \infty.
\ee

As above, the cocycle property with $t=k$, $s=t_{n^{(k)}}$ and $\omega$ being replaced by $\vartheta_{-t_{n^{(k)}}} \omega$
yields
\be\label{equ:varphi2}
\varphi( t_{n^{(k)}}, \vartheta_{- t_{n^{(k)}}  } \omega)
= \varphi(k,\vartheta_{-k}\omega) \varphi( t_{n^{(k)}} - k, \vartheta_{ - t_{n^{(k)}}  } \omega), \quad k \in \nN.
\ee
Hence, from \eqref{def:ymk} and by applying the last part of Theorem \ref{thm:limit}, we get
\be\label{equ:weaklim}
\begin{aligned}
 \vy_0 &= \mbox{w} - \lim_{ n^{(k)} \rightarrow \infty}
 \varphi( t_{n^{(k)}}, \vartheta_{ - t_{n^{(k)}}  } \omega) \x_{ n^{(k)}} \\
       &= \mbox{w} - \lim_{ n^{(k)} \rightarrow \infty}
       \varphi(k,\vartheta_{-k}\omega) \varphi( t_{n^{(k)}} - k, \vartheta_{ - t_{n^{(k)}}  } \omega) \x_{ n^{(k)}} \\
       &= \varphi(k,\vartheta_{-k}\omega)
         (  \mbox{w} - \lim_{ n^{(k)} \rightarrow \infty}
           \varphi( t_{n^{(k)}} - k, \vartheta_{ - t_{n^{(k)}}  } \omega) \x_{ n^{(k)}} ) \\
       &= \varphi(k,\vartheta_{-k}\omega) \vy_{-k},\\
\end{aligned}
\ee
where w-$\lim$ denotes the limit in the weak topology on $H$. The same proof yields a more general property:
\[
\varphi(j, \vartheta_{-k} \omega) \vy_{-k} = \vy_{-k+j} \mbox {  if } 0 \le j \le k.
\]

Before continuing with the proof, let us point out that \eqref{equ:weaklim} means precisely that
$\vy_0 = \vu(0,-k;\omega,\vy_{-k})$, where $\vu$ is defined in \eqref{def:uv}.

{\bf Step III.}
Proof of \eqref{equ:limsup}. From now on, unless explicitly stated, we fix $k \in \nN$, and we will consider
problem \eqref{sNSEs} on the time interval $[-k,0]$. From \eqref{def:uv} and \eqref{equ:varphi2}, with $t=0$
and $s=-k$, we have
\be\label{equ:8_15}
\begin{aligned}
&\|\varphi( t_{n^{(k)}}, \vartheta_{- t_{n^{(k)}}  } \omega)\x_{ n^{(k)} } - \ovz(0)\|^2  \\
& \qquad = \| \varphi(k,\vartheta_{-k}\omega) \varphi( t_{n^{(k)}} - k, \vartheta_{ - t_{n^{(k)}}  }  \omega)\x_{ n^{(k)} }
       - \ovz(0)\|^2 \\
& \qquad = \| \vv(0,-k;\omega, \varphi( t_{n^{(k)}} - k, \vartheta_{ - t_{n^{(k)}}  }  \omega)\x_{ n^{(k)}} - \ovz(-k))\|^2.
\end{aligned}
\ee

Let $\vv$ be the solution to \eqref{equ:ode} on $[-k,\infty)$ with $\ovz=\ovz_\alpha(\cdot,\omega)$
and the initial condition at time $-k$:
$\vv(-k) = \varphi( t_{n^{(k)}} - k, \vartheta_{ - t_{n^{(k)}} } \omega)\x_{n^{(k)}} - \ovz(-k)$.
In other words,
\[
 \vv(s) =
 \vv\big(s,-k;\omega, \varphi( t_{n^{(k)}} - k, \vartheta_{ - t_{n^{(k)}} } \omega)\x_{n^{(k)}} - \ovz(-k) \big),
 \quad s \ge -k.
\]
From \eqref{equ:8_15} and \eqref{equ:8_11} with $t=0$ and $\tau=-k$ we infer that
\be\label{equ:8_16}
\begin{aligned}
&\|\varphi(t_{n^{(k)}},\vartheta_{ - t_{n^{(k)}} } \omega)\x_{n^{(k)}} - \ovz(0)\|^2 =
e^{-\nu \lambda_1 k}\|\varphi( t_{n^{(k)}} - k, \vartheta_{ - t_{n^{(k)}} } \omega)\x_{n^{(k)}} - \ovz(-k)\|^2 \\
&+ 2 \int_{-k}^0 e^{\nu \lambda_1 s} (b(\vv(s),\ovz(s),\vv(s)) + \langle \vg(s),\vv(s) \rangle
+ \langle \vf, \vv(s) \rangle - [\vv(s)]^2) ds.
\end{aligned}
\ee

It is enough to find a nonnegative function $h \in L^1(-\infty,0)$ such that
\be\label{equ:8_17}
\limsup_{ n^{(k)} \rightarrow \infty } \|\varphi(t_{n^{(k)}},\vartheta_{ - t_{n^{(k)}} } \omega)\x_{n^{(k)}} - \ovz(0)\|^2 \le \int_{-\infty}^{-k} h(s) ds + \|\vy_0 - \ovz(0)\|^2.
\ee
For, if we define the diagonal process $(m_j)_{j=1}^\infty$ by $m_j = j^{(j)}$, $j\in \nN$, then for each
$k\in \nN$, the sequence $(m_j)_{j=k}^\infty$ is a subsequence of the sequence $(n^{(k)})$ and hence by
\eqref{equ:8_17},
$
 \limsup_j\|\varphi(t_{m_j},\vartheta_{ - t_{m_j} } \omega)\x_{m_j} - \ovz(0)\|^2
 \le \int_{-\infty}^{-k} h(s) ds + \|\vy_0 - \ovz(0)\|^2.
$ Taking the $k\rightarrow \infty$ limit in the last inequality we infer that
\[
 \limsup_j\|\varphi(t_{m_j},\vartheta_{ - t_{m_j} } \omega)\x_{m_j} - \ovz(0)\|^2
  \le \|\vy_0 - \ovz(0)\|^2,
\]
which proves claim \eqref{equ:limsup}.

{\bf Step IV.} Proof of \eqref{equ:8_17}. We begin with estimating the first term on the RHS
of \eqref{equ:8_16}. If $-t_{n^{(k)}} < -k$, then by \eqref{def:uv} and \eqref{equ:8_10} we infer that
\be\label{equ:8_18}
\begin{aligned}
&\|\varphi(t_{n^{(k)}} - k, \vartheta_{-t_{n^{(k)}} } \omega) \x_{n^{(k)}} - \ovz(-k)\|^2 \\
&\qquad= \|\vv(-k,-t_{n^{(k)}};\vartheta_{-k}\omega, \x_{n^{(k)}} - \ovz(-t_{n^{(k)}})\|^2 e^{-\nu \lambda_1 k}  \\
&\le e^{-\nu \lambda_1 k}
    \big\{ \| \x_{n^{(k)}} - \ovz(-t_{n^{(k)}}) \|^2
             \exp\left( -\nu\lambda_1(t_{n^{(k)}} -k)+ \frac{27C^4}{16\nu^3}
        \int_{-t_{n^{(k)}} }^{-k} \|\ovz(s)\|^4_{\mathbb{L}^4} ds\right) \\
& \qquad + \frac{3}{\nu}
   \int_{-t_{n^{(k)}} }^{-k} [\|\vg(s)\|^2_{V^\prime} + \|\vf\|^2_{V^\prime} ]
       \exp\left(
     -\nu \lambda_1(-k-s) + \frac{27C^4}{16\nu^3} \int_s^{-k}
          \|\ovz(\zeta)\|^4_{\mathbb{L}^4}d\zeta \right)   \big\}  \\
&\le 2I_{ n^{(k)}} + 2 II_{n^{(k)}}+ \frac{3}{\nu} III_{n^{(k)}}+ \frac{3}{\nu} IV_{n^{(k)}},
\end{aligned}
\ee
where
\begin{align*}
I_{ n^{(k)}}  &=\|\x_{n^{(k)}}\|^2
  \exp\left( -\nu \lambda_1 t_{n^{(k)}} +
             \frac{27C^4}{16\nu^3} \int_{-t_{n^{(k)}} }^{-k} \|\ovz(s)\|^4_{\mathbb{L}^4} ds  \right)  \\
II_{ n^{(k)}} &= \|\ovz( t_{n^{(k)}} )\|^2
       \exp\left( -\nu \lambda_1 t_{n^{(k)}} +
             \frac{27C^4}{16\nu^3} \int_{-t_{n^{(k)}} }^{-k} \|\ovz(s)\|^4_{\mathbb{L}^4} ds  \right)  \\
III_{ n^{(k)}}&= \int_{-t_{n^{(k)}} }^{-k} \|\vg(s)\|^2_{V^\prime}
   \exp\left( -\nu\lambda_1 s + \frac{27C^4}{16\nu^3}
                      \int_s^{-k}\|\ovz(\zeta)\|^4_{\mathbb{L}^4}d\zeta  \right) \\
IV_{ n^{(k)}} &=\int_{-t_{n^{(k)}} }^{-k} \|\vf(s)\|^2_{V^\prime}
          \exp\left( -\nu\lambda_1 s + \frac{27C^4}{16\nu^3}
                      \int_s^{-k}\|\ovz(\zeta)\|^4_{\mathbb{L}^4}d\zeta  \right)
\end{align*}

First we will find a nonnegative function $h \in L^1(-\infty,0)$ such that
\be\label{equ:8_19}
\limsup_{ n^{(k)} \rightarrow \infty} \| \varphi( t_{n^{(k)}} - k,\vartheta_{-t_{n^{(k)}}}\omega)\x_{n^{(k)}} - \ovz(-k)\|^2
  e^{-\nu \lambda_1 k} \le \int_{-\infty}^{-k} h(s) ds, \quad  k \in \nN.
\ee
This will be accomplished as soon as we prove the following four lemmas.

\begin{lemma}\label{lem:8_4}
$\limsup_{ n^{(k)} \rightarrow \infty } I_{n^{(k)}} = 0.$
\end{lemma}

\begin{lemma}\label{lem:8_5}
$\limsup_{ n^{(k)} \rightarrow \infty } II_{n^{(k)}} = 0.$
\end{lemma}

\begin{lemma}\label{lem:8_6}
$\int_{-\infty}^0 \|\vg(s)\|^2_{V^\prime}
  \exp\left( -\nu\lambda_1 s + \frac{27C^4}{16\nu^3}
                      \int_s^{0}\|\ovz(\zeta)\|^4_{\mathbb{L}^4}d\zeta  \right) < \infty.$
\end{lemma}

\begin{lemma}\label{lem:8_7}
$\int_{-\infty}^0  \exp\left( -\nu\lambda_1 s + \frac{27C^4}{16\nu^3}
                      \int_s^{0}\|\ovz(\zeta)\|^4_{\mathbb{L}^4}d\zeta  \right) < \infty.$
\end{lemma}

\begin{proof}[Proof of Lemma~\ref{lem:8_4}]
We recall that for $\alpha \in \nN$, $\ovz(t) = \ovz_\alpha(t)$, $t \in \R$, being the
Ornstein-Uhlenbeck process from subsection~\ref{sec:OUprocess}, one has
\[
 \E\|\ovz(0)\|^4_X = \E \|\ovz_\alpha(0)\|^4_X < \frac{8\nu^4\lambda_1}{27 C^4}.
\]
Let us recall that the space $\hat{\Omega}(\xi,E)$ was constructed in such a way that
\[
 \lim_{ n^{(k)} \rightarrow \infty} \frac{1}{-k-(-t_{n^{(k)}})} \int_{t_{n^{(k)}} }^{-k}
  \| \ovz_\alpha(s)\|^4_X ds = \E\|\ovz(0)\|^4_X < \infty.
\]
Therefore, since the embedding $X \hookrightarrow \mathbb{L}^4(\bS)$ is a contraction, we have for $n^{(k)}$ sufficiently large,
\be\label{equ:8_20}
 \frac{27C^4}{16\nu^3} \int_{t_{n^{(k)}} }^{-k}\| \ovz_\alpha(s)\|^4_{\mathbb{L}^4} ds < \frac{\nu\lambda_1}{2}(t_{n^{(k)}} - k).
\ee
Since the set $B$ is bounded in $H$, there exists $\rho_1 >0$ such that for all $n^{(k)}$,
$\|\x_{n^{(k)}}\| \le \rho_1$. Hence
\be\label{equ:8_21}
 \limsup_{ n^{(k)} \rightarrow \infty} \|\x_{n^{(k)}}\|^2
\exp\left( -\nu \lambda_1 t_{n^{(k)}} +
              \frac{27C^4}{16\nu^3} \int_{-t_{n^{(k)}} }^{-k} \|\ovz(s)\|^4_{\mathbb{L}^4} ds  \right)
              \le \limsup_{ n^{(k)} \rightarrow \infty} \rho_1^2 e^{ -\frac{\nu\lambda_1}{2}(t_{n^{(k)}} - k)  } = 0.
\ee
\end{proof}

\begin{proof}[Proof of Lemma~\ref{lem:8_7}]
We denote by
\[
p(s) = \nu \lambda_1 s + \frac{27C^4}{16\nu^3} \int_s^0 \|\ovz(s)\|^4_{\mathbb{L}^4}.
\]
As in the proof of
Lemma~\ref{lem:8_4} we have, for $s\le s_0$, $p(s) < \frac{\nu \lambda_1}{2} s$. Hence $\int_{-\infty}^0 e^{p(s)}ds < \infty$,
as required.
\end{proof}

\begin{proof}[Proof of Lemma~\ref{lem:8_5}]
Because of \eqref{equ:zhatgrow}, we can find $\rho_2 \ge 0$ and $s_0 <0$, such that,
\be\label{equ:8_22}
\max\left( \frac{\|\ovz(s)\|} { |s|} , \frac{\|\ovz(s)\|_{V^\prime}}{|s|}, \frac{\|\ovz(s)\|_{\mathbb{L}^4}}{|s|} \right) \le \rho_2,
\quad \mbox{ for } s \le s_0.
\ee
Hence by \eqref{equ:8_20} we infer that
\be\label{equ:8_23}
\begin{aligned}
&\limsup_{ n^{(k)} \rightarrow \infty} \|\ovz(-t_{n^{(k)}})\|^2
       \exp \left( \int_{-t_{n^{(k)}} }^{-k} (-\nu \lambda_1
 + \frac{27C^4}{16\nu^3}\|\ovz(s)\|^4_{\mathbb{L}^4}) ds  \right) \\
& \qquad \le \limsup_{ n^{(k)} \rightarrow \infty} \frac{\|\ovz(-t_{n^{(k)}})\|^2 }{|t_{n^{(k)}}|^2}
             \limsup_{ n^{(k)} \rightarrow \infty} |t_{n^{(k)}}|^2  e^{-\frac{\nu\lambda_1}{2}( t_{n^{(k)}}-k ) }  \le 0.
\end{aligned}
\ee
This concludes the proof of Lemma~\ref{lem:8_5}.
\end{proof}

\begin{proof}[Proof of Lemma~\ref{lem:8_6}]
Since $\|\vg(s)\|^2_{V^\prime} = \|\alpha \ovz(s) + 2\B(\ovz(s))\|^2_{V^\prime} \le 2 \alpha^2 \|\ovz(s)\|^2_{V^\prime} + 2C \|\ovz(s)\|^4_{\mathbb{L}^4}$,
we only need to show that the integrals
\[
\int_{-\infty}^0 \|\ovz(s)\|^4_{\mathbb{L}^4}
\exp \left(\nu\lambda_1 s + \frac{27C^4}{16\nu^3}\int_s^0\|\ovz(\zeta)\|^4_{\mathbb{L}^4}d\zeta \right) ds
\]
and
\[
\int_{-\infty}^0 \|\ovz(s)\|^2_{V^\prime}
\exp\left(\nu\lambda_1 s + \frac{27C^4}{16\nu^3}\int_s^0\|\ovz(\zeta)\|^4_{\mathbb{L}^4}d\zeta\right) ds
\]
are finite.

It is enough to consider the case of $\|\ovz(s)\|^4_{\mathbb{L}^4}$ since the proof will be similar for the remaining case.
Reasoning as in \eqref{equ:8_20}, we can find $t_0 \ge 0$ such that for $ t \ge t_0$,
\[
 \int_{-t}^{-t_0} \left( -\nu\lambda_1 + \frac{27C^4}{6\nu^3} \|\ovz(\zeta)\|_{\mathbb{L}^4}^4   \right) d\zeta
  \le -\frac{\nu\lambda_1}{2} (t-t_0).
\]
Taking into account the inequality \eqref{equ:8_22}, we have $\|\ovz(t)\| \le \rho_2(1+|t|)$, $t \in \R$.
Therefore, with
$\rho_3 := \exp( \int_{-t_0}^0 (-\nu\lambda_1 + \frac{27C^4}{16\nu^3}
\|\ovz(\zeta)\|^4_{\mathbb{L}^4})d\zeta$, we have
\[
\begin{aligned}
&\int_{-\infty}^{-t_0} \|\ovz(s)\|^4_{\mathbb{L}^4}
   \exp\left( \int_{s}^0 (\nu\lambda_1  +
  \frac{27C^4}{16\nu^3}\|\ovz(\zeta)\|^4_{\mathbb{L}^4})d\zeta \right) ds \\
& \qquad = \rho_3 \int_{-\infty}^{-t_0} \|\ovz(s)\|^4_{\mathbb{L}^4}
   \exp\left(\int_{s}^{-t_0} (\nu\lambda_1
  + \frac{27C^4}{16\nu^3} \|\ovz(\zeta)\|^4_{\mathbb{L}^4})d\zeta\right) ds \\
& \qquad \le \rho_2^4 \rho_3 e^{\nu\lambda_1 t_0/2}  \int_{-\infty}^{t_0} |s|^4 e^{\nu\lambda_1 s/2} ds < \infty.
\end{aligned}
\]
By the continuity of all relevant functions, we can let $t_0 \rightarrow 0$ to get the result.
\end{proof}

Therefore, the proof of \eqref{equ:8_19} is concluded, and it only remains to finish the proof of \eqref{equ:8_17}.
Let us denote by
\begin{align*}
\vv_\nk(s) &= \vv(s,-k;\omega, \varphi(t_\nk - k, \vartheta_{-t_\nk}\omega) \x_\nk - \ovz(-k)), \;\; s \in (-k,0),\\
\vv_k(s) &= \vv(s,-k;\omega, \vy_{-k} - \ovz(-k)), \;\; s \in (-k,0).
\end{align*}

From \eqref{def:ymk} and the last part of Theorem \ref{thm:limit} we infer that
\be\label{equ:8_24}
\vv_{\nk} \rightarrow \vv_k \mbox{ weakly in } L^2(-k,0;V).
\ee

Since $e^{\nu\lambda_1 \cdot} \vg, \; e^{\nu \lambda_1 \cdot}
\vf \in L^2(-k,0;V^\prime)$, we get
\be\label{equ:8_25}
 \lim_{\nk \rightarrow \infty} \int_{-k}^0 e^{\nu\lambda_1 s} \langle \vg(s), \vv_{\nk}(s)\rangle ds
    = \int_{-k}^0 e^{\nu\lambda_1 s} \langle \vg(s),\vv_k(s)\rangle ds
\ee
and
\be\label{equ:8_26}
 \lim_{\nk \rightarrow \infty} \int_{-k}^0 e^{\nu\lambda_1 s} \langle \vf, \vv_{\nk}(s)\rangle ds
    = \int_{-k}^0 e^{\nu\lambda_1 s} \langle \vf,\vv_k(s)\rangle ds.
\ee

On the other hand, using the same methods as those in the proof of Theorem~\ref{thm:existence}, there
exists a subsequence of $\{\vv_\nk\}$, which, for the sake of simplicity of notation, is denoted as the
old one which satisfies
\be\label{equ:8_27}
 \vv_{\nk} \rightarrow \vv_k \mbox{ strongly in }
 L^2(-k,0;\mathbb{L}^2_{\loc}(\bS)).
\ee

Next, since $\ovz(t)$ is an $\mathbb{L}^4$-valued process, so is $e^{\nu\lambda_1 t} \ovz(t)$. Thus by
\cite[Corollary~4.1]{BrzGolLeG15},
\eqref{equ:8_24} and \eqref{equ:8_27}, we infer that
\be\label{equ:8_28}
\begin{aligned}
&\lim_{\nk \rightarrow \infty} \int_{-k}^0 e^{\nu \lambda_1 s} b(\vv_{\nk}(s), \ovz(s),\vv_{\nk}(s)) ds \\
&\qquad =  \int_{-k}^0 e^{\nu \lambda_1 s} b(\vv_k(s),\ovz(s),\vv_k(s)) ds.
\end{aligned}
\ee

Moreover, since the norms $[\cdot]$ and $\|\cdot\|_V$ are equivalent on $V$, and since for any $s \in (-k,0]$,
$e^{-\nu k \lambda_1} \le e^{\nu \lambda_1 s} \le 1$, $(\int_{-k}^0 e^{\nu \lambda_1 s} [\cdot]^2 ds)^{1/2}$
is a norm in $L^2(-k,0;V)$ equivalent to the standard one. Hence, from \eqref{equ:8_24} we obtain,
\[
\int_{k}^0 e^{\nu \lambda_1 s} [\vv_k(s)]^2 ds \le
\liminf_{\nk\rightarrow \infty} \int_{-k}^0 e^{\nu\lambda_1s} [\vv_{\nk}(s)]^2 ds.
\]
In other words,
\be\label{equ:8_29}
\limsup_{\nk \rightarrow \infty}
\left(-\int_{-k}^0 e^{\nu\lambda_1 s} [\vv_{\nk}(s)]^2 ds \right)\le - \int_{-k}^0 e^{\nu\lambda_1 s} [\vv_k(s)]^2 ds.
\ee

From \eqref{equ:8_16}, \eqref{equ:8_19}, \eqref{equ:8_28} and \eqref{equ:8_29} we infer that
\be\label{equ:8_30}
\begin{aligned}
&\limsup_{\nk \rightarrow \infty} \|\varphi(t_{\nk},\vartheta_{-t_{\nk}}\omega)\x_{\nk} - \ovz(0)\|^2 \\
&\qquad \le \int_{-\infty}^{-k} h(s) ds + 2\int_{-k}^0 e^{\nu\lambda_1 s}
\big\{  b(\vv_k(s),\ovz(s),\vv_k(s)) \\
&\qquad \qquad \qquad+ \langle \vg(s),\vv_k(s) \rangle + \langle\vf,\vv_k(s)\rangle - [\vv_k(s)]^2 \big\} ds
\end{aligned}
\ee

On the other hand, from \eqref{equ:weaklim} and \eqref{equ:8_11}, we have
\be\label{equ:8_31}
\begin{aligned}
\|\vy_0 - \ovz(0)\|^2 &= \|\varphi(k,\vartheta_{-k}\omega)\vy_k - \ovz(0)\|^2 = \|\vv(0,-k;\omega,\vy_k-\ovz(-k))\|^2 \\
 &=\|\vy_k - \ovz(-k)\|^2 e^{-\nu\lambda_1 k} + 2 \int_{-k}^0 e^{\nu \lambda_1 s} \big\{ \langle \vg(s),\vv_k(s) \rangle\\
 &+ b(\vv_k(s),\ovz(s),\vv_k(s)) + \langle \vf,\vv_k(s) \rangle - [\vv_k(s)]^2 \big\} ds.
\end{aligned}
\ee

Hence, by combining \eqref{equ:8_30} with \eqref{equ:8_31}, we get
\[
\begin{aligned}
&\limsup_{\nk\rightarrow \infty} \|\varphi(t_{\nk}, \vartheta_{-t_{\nk}}\omega)\x_{\nk} - \ovz(0)\|^2 \\
&\quad\le \int_{-\infty}^{-k} h(s) ds + \|\vy_0 - \ovz(0)\|^2 - \|\vy_k - \ovz(-k)\|^2 e^{-\nu\lambda_1 k} \\
&\quad\le \int_{-\infty}^{-k} h(s) ds + \|\vy_0 - \ovz(0)\|^2,
\end{aligned}
\]
which proves \eqref{equ:8_17}, and hence the proof of Proposition~\ref{prop:ACprop} is finished.
\end{proof}
\section{Invariant measure}
In this section we consider the existence of an invariant measure.
The main result in this section, i.e. Theorem \ref{thm-inv-measure} is a direct
consequence of Corollary 4.4 \cite{Crauel+F_1994} and  our Theorem \ref{thm-main}
about the existence of an attractor for the  RDS generated by  the stochastic Navier-Stokes equations
\eqref{sNSEs}.

Let $\varphi$ be the RDS corresponding to the SNSEs \eqref{sNSEs} and defined in \eqref{def:varphi}.
 We define the transition operator
$P_t$ by a standard formula. For $g \in \mathcal{B}_b(\mathrm{H})$, we put
\begin{eqnarray}
\label{eq-trans-sem} P_tg(\x)&=&\int_{\Omega} \left[g(\varphi(t,\omega,\x))\right] d \mathbb{P}(\omega)
,\quad \x \in {H}.
\end{eqnarray}
As in  \cite[Proposition 3.8]{BrzLi06} we have  the following result whose proof  is simply a repetition of the proof from \cite{BrzLi06}
\begin{proposition}\label{prop-Feller}
The family $(P_t)_{t \geq 0}$ is Feller, i.e. $P_tg \in C_b(H)$ if $g\in
C_b(H)$. Moreover,  for any $g\in
C_b(\mathrm{X})$,  $P_tg(\x) \to g(\x)$ as $t\todown 0$.
\end{proposition}
Following \cite{Crauel+F_1994} one can prove that $\varphi$ is a Markov RDS,
i.e. $P_{t+s}=P_tP_s$ for all $t,s \geq 0$. Hence by \cite[Corollary 3.10]{BrzLi06} which says that
 a time-continuous and continuous   asymptotically compact, Markov RDS  $\varphi$  admits a Feller  invariant probability measure
 $\mu$, i.e. a Borel probability measure $\mu$
\begin{eqnarray}
\label{eq-inv-meas}
P_t^\ast\mu=\mu, \; t \ge 0,
\end{eqnarray}
 where
 \[P_t^\ast \mu (\Gamma)=\int_\mathrm{H}P_t(x,\Gamma)
\, \mu(dx),\quad \Gamma \in \mathcal{B}(\mathrm{H}),\]
 and $P_t(x,\cdot)$ is the
transition probability, $P_t(x,\Gamma)= P_t1_\Gamma(x)$, $x \in H$. \\
A Feller invariant probability measure for a Markov RDS $\varphi$ on $\mathrm{H}$ is,
by definition, an invariant probability measure for the semigroup $(P_t)_{t \geq 0}$
defined by \eqref{eq-trans-sem}. Therefore, we obtain the following result.

\begin{theorem}\label{thm-inv-measure} There exists an invariant
measure for the stochastic NSE \eqref{sNSEs}.
\end{theorem}

\begin{remark}\label{rem-inv-measure} We believe that the  uniqueness of an invariant measure for nondegenerate noise will follow from the classical procedure based on Doob's Theorem, see e.g. Seidler's paper \cite{Seidler_1997} and references therein. If the noise is degenerate and spatially smooth, it seems that the results from a recent paper by Hairer and Mattingly \cite{Hairer+Matt_2011} should  be applicable
in our setting. In particular, \cite[Theorem 8.4]{Hairer+Matt_2011},  which gives a sufficient
conditions for uniqueness in terms of controllability, should be applicable. Details
will be published elsewhere. One should point out that these authors use the "vorticity" formulation and their initial data belongs to
the $L^2$ space.  This corresponds to our approach with the initial
data belonging to the finite enstrophy space $H^1$. However, we work in the space of finite energy, which seems to be physically more natural.   On the other hand,
verifying the  sufficient conditions could be more challenging. For
the NSE without the Coriolis force this problems has been
investigated in  \cite{Agra+Sar_2008}. Corresponding NSE with the
Coriolis force study is postponed till the next publication.
\end{remark}


\end{document}